\documentclass[leqno,final]{siamltex}
\usepackage{graphicx} 
\usepackage{amsmath,amstext,amssymb,bm}
\usepackage{leftidx}
\numberwithin{equation}{section}
\usepackage{xcolor}
\usepackage{soul}
\usepackage{tikz}
\usetikzlibrary{shapes,arrows}
\usepackage{mathrsfs}

\newtheorem{remark}{Remark}[section]

\allowdisplaybreaks[4]

\newcommand{\eps}{\varepsilon}
\newcommand{\Ome}{{\Omega}}

\newcommand{\R}{\mathbb{R}}

\newcommand{\N}{\mathbb{N}}

\begin{document}

\title{A new theory of fractional differential calculus\thanks{This work was partially supported by the NSF grant 1620168.} }
\markboth{X. FENG and M. SUTTON}{A NEW THEORY OF WEAK FRACTIONAL CALCULUS}

%
	
\author{Xiaobing Feng\thanks{Department of Mathematics, The University of Tennessee, 
Knoxville, TN 37996. U.S.A. (xfeng@math.utk.edu).}
\and{Mitchell Sutton}\thanks{Department of Mathematics, The University of Tennessee, 
Knoxville, TN 37996. U.S.A. (msutto11@vols.utk.edu).} }

\date{}

\maketitle
 
\thispagestyle{empty}

\begin{abstract}
    This paper presents a self-contained new theory of weak fractional differential calculus in one-dimension.
    The crux of this new theory is the introduction of a weak fractional derivative notion which is a natural generalization of integer order weak derivatives; it also helps to unify multiple existing fractional derivative definitions and characterize what functions are fractionally differentiable. Various calculus rules including a fundamental theorem calculus, product and chain rules, and integration by parts formulas are established for weak fractional derivatives. 
    Additionally, relationships with classical fractional derivatives and detailed characterizations  of weakly fractional differentiable functions are also established.
    Furthermore, the notion of weak fractional derivatives is also systematically extended to general distributions instead of only to some special distributions. 
    This new theory lays down a solid theoretical foundation for systematically and rigorously developing new theories of fractional Sobolev spaces, fractional calculus of variations, and fractional PDEs as well as their numerical solutions in subsequent works. 
    This paper is a concise presentation of the materials of Sections 1-4 and 6 of reference 
    \cite{Feng_Sutton}. 
    
\end{abstract}




 

\section{Introduction}\label{sec-1}

    Similar to the classical integer order calculus, the classical fractional order calculus also consists of two integral parts, namely, the fractional order integral calculus and the fractional order differential calculus. It is concerned with studying their properties/rules and the interplay between the two notions, which is often characterized by the so-called {\em Fundamental Theorem of Calculus}. Fractional calculus also has had a long history, which can be traced back to L'H\^{o}pital (1695), Wallis (1697), Euler (1738), Laplace (1812), Lacroix (1820), Fourier (1822), Abel (1823), Liouville (1832), Riemann (1847), Leibniz (1853), Gr\"unwald (1867), Letnikov (1868) and many others. We refer the reader to \cite{Guo, Podlubny, Samko} and the references therein for a detailed exposition about the history of the classical fractional calculus. 
 
    In the past twenty years fractional calculus and fractional (order) differential equations have garnered a lot of interest and attention both from the PDE community (in the name of nonlocal PDEs) and in the applied mathematics and scientific communities. Besides the genuine mathematical interest and curiosity, this trend has also been driven by intriguing scientific and engineering applications which give rise to fractional order differential equation models to better describe the (time) memory effect and the (space) nonlocal phenomena (cf. \cite{Du2019,Guo,Hilfer,Kilbas, Meerschaert} and the references therein). It is the rise of these applications that revitalizes the field of fractional calculus and fractional differential equations and calls for further research in the field, including to develop new numerical methods for solving various fractional order problems. 
 
    Although a lot of progress has been achieved in the past twenty years in the field of fractional calculus and fractional differential equations, many fundamental issues remain to be addressed. For a novice in the field, one would immediately be clogged and confused by many (non-equivalent) definitions of fractional derivatives. The immediate ramification of the situation is the difficulty for building/choosing ``correct" fractional models to study analytically and to solve numerically. Moreover, compared to the classical integer order calculus, the classical fractional calculus still has many missing components. For example, many basic calculus rules (such as product and chain rules) are not completely settled, the physical and geometric meaning of fractional derivatives are not fully understood, and a thorough characterization of the fractional differentiability seems still missing. Furthermore, at the differential equation (DE) level, the gap between the integer order and fractional order cases is even wider. For a given integer order DE, it would be accustomed for one to interpret the derivatives in the DE as weak derivatives and the solution as a weak solution. However, there is no parallel weak derivative and solution theory in the fractional order case so far. As a result, various solution concepts and theories, which may not be equivalent, have been used for fractional order DEs, especially, for fractional (order) partial differential equations (FPDEs). The non-equivalence of solution concepts may cause confusions and misunderstanding of the underlying fractional order problem.
 
    The primary goal of this paper is to develop a new weak fractional differential calculus theory, starting from introducing the definition, proving various 
    calculus rules, to establishing the {\em Fundamental Theorem of Weak Fractional Calculus} (FTwFC).
    These results lay down the ground work for developing a new fractional Sobolev space theory
    in a companion paper \cite{Feng_Sutton1a}. 
    Together they provide a first step/attempt in achieving the overreaching goal of providing the missing components to, and to expand the reach of, the classical fractional calculus and fractional differential equations, which will be continued in subsequent works \cite{Feng_Sutton2,Feng_Sutton3}. 
 
    The remainder of this paper is organized as follows. In Section \ref{sec-2},  we recall the definitions of some well-known classical fractional derivatives and a few relevant 
    properties, which all can be found in the beginning chapters of the reference \cite{Samko}
    (also see \cite[Section 2]{Feng_Sutton}). Additionally, we present an alternative perspective of the classical fractional derivatives. This is to introduce the so-called {\em Fundamental Theorem of Classical Fractional Calculus} (FTcFC) and a new interpretation/definition of 
    the classical
    Riemann-Liouville fractional derivatives. Such a viewpoint is essential to developing the 
    weak fractional calculus theory in this paper and the fractional Sobolev space theory in \cite{Feng_Sutton1a}.
    In Section \ref{sec-4}, we first introduce the notion of weak fractional derivatives using integration by parts and special test functions, which is analogous to the notion of integer order weak derivatives. It is proved that weak fractional derivatives inherit the fundamental properties of classical fractional derivatives and the generality eliminates the need for numerous definitions as seen in the classical theory. After having proved a characterization result, we then establish the FTwFC, product and chain rules, and integration by parts formulas. Many of these results and their proof techniques will look familiar to the informed reader because they are adapted and refined  from those used in the integer order weak differential calculus theory (cf. \cite{Adams, Evans, Meyers}). The desired differential calculus components are proved for both left and right weak fractional derivatives and covers both finite and infinite domain cases. 
    In Section \ref{sec-6}, we extend the notion of weak fractional derivatives to distributions. Unlike the existing fractional derivative definitions which only apply to a certain subset of distributions, we aim to define weak fractional derivatives for general distributions.
    Due to the pollution effect of the fractional derivatives, the main difficulty to overcome 
    is to design a good domain extension for a given distribution, which is achieved by using 
    a partition of the unity idea in this section. 
    Finally, the paper is concluded by a short summary and a few concluding remarks given in Section \ref{sec-7}.

\section{Preliminaries}\label{sec-2}
    In this section we first recall definitions of classical fractional integrals and derivatives 
    without stating their well-known properties, we refer 
    the interested reader to \cite{Feng_Sutton} for a collection of useful properties, 
    and to \cite{Samko} for a more extensive collection and their detailed proofs. 
%
    We then present a couple lesser known properties of classical fractional derivatives which will be crucially and repeatedly used in the subsequent sections. One of which is the pollution behavior of classical fractional derivatives of compactly supported smooth function, the other is an equivalent definition of the Riemann-Liouville fractional derivatives based 
    on the FTcFC.

    Since the definitions of fractional integrals and derivatives are domain-dependent, it will be imperative for us to separate the cases when the domain is finite or infinite. In this section and in the later sections, we shall consider both finite interval $\Omega=(a,b)$ for $\infty < a < b < \infty$ and the infinite interval $\Omega=\R:=(-\infty, \infty)$. 

    Throughout this paper, $\Gamma: \R\to \R$ denotes the standard Gamma function and $\N$ stands for the set of all positive integers. In addition, $C$ will be used to denote a generic positive constant which may be different at different locations and $f^{(n)}$ denotes the $n$th order classical derivative of $f$ for $n\in\N$. Unless stated otherwise, all integrals $\int_a^b \varphi(x)\, dx$ are understood as Riemann integrals in this section. 
    
    \subsection{Definitions of Classical Fractional Integrals and Derivatives}\label{sec-2.1}
     In this subsection, we recall the definitions 
        of classical Riemann-Liouville, Caputo, Gr\"{u}nwald-Letnikov, and Fourier fractional integrals and derivatives. 

        \subsubsection{\bf Definitions on a Finite Interval}\label{sec-2.1.1}
            Historically, the integral calculus was invented before the differential calculus in the classical Newton-Leibniz (integer) calculus  and the two are intimately connected through the well known {\em Fundamental Theorem of Calculus} (or Newton-Leibniz Theorem). It is interesting to note that the same is true for the classical fractional calculus. Indeed, in order to give a definition of fractional derivatives,  we first need to recall the definition of fractional integrals.

            \begin{definition}[cf. \cite{Samko}] \label{def2.1}
                Let $\sigma>0$ and $f:[a,b] \rightarrow \R$. The $\sigma$  order left Riemann-Liouville fractional integral of $f$ is defined by
                \begin{align} \label{left_RL_int}
                    {_{a}}{I}{^{\sigma}_{x}} f(x) : = \dfrac{1}{\Gamma(\sigma)} \int_{a}^{x} \dfrac{f(y)}{(x - y)^{1 - \sigma}} \, dy \qquad\forall x\in [a,b], 
                \end{align}
                and the $\sigma$ order right Riemann-Liouville fractional integral of $f$ is defined by 
                \begin{align}\label{right_RL_int}
                    {_{x}}{I}{^{\sigma}_{b}} f(x) : = \dfrac{1}{\Gamma(\sigma)} \int_{x}^{b} \dfrac{f(y)}{(y-x)^{1 - \sigma}}\, dy\qquad \forall x\in [a,b].
                \end{align}
            \end{definition}
            ${_{a}}{I}{^{\sigma}_{x}}$ and $ {_{x}}{I}{^{\sigma}_{b}}$ are respectively called the left and right Riemann-Liouville fractional integral operators. We also set ${^{-}}{I}{^{\sigma}} := {_{a}}{I}{^{\sigma}_{x}}$ and ${^{+}}{I}{^{\sigma}} := {_{x}}{I}{^{\sigma}_{b}}$.

            \begin{remark}
                It is well known (\cite{Samko}) that both $ {_{a}}{I}{^{\sigma}_{x}}$ and ${_{x}}{I}{^{\sigma}_{b}}$ are convolution-type operators (with different kernel functions). 
            \end{remark}

            With the help of the above fractional integrals,  the definitions of two popular Riemann-Liouville fractional derivatives are given below.  

            \begin{definition}[cf. \cite{Samko}] \label{def2.2} 
                Let $n-1 < \alpha < n$ and $f : [a,b] \rightarrow \R$. The  $\alpha$ order left Riemann-Liouville fractional derivative of $f$ is defined by
                \begin{align}\label{left_RL_derivative}
                    {_{a}}{{D}}{_{x}^{\alpha}} f(x):&= \dfrac{d^{n}}{dx^n} \Bigl( {_{a}}{I}{^{n-\alpha}_{x}}f(x) \Bigr) \\
                    &=\dfrac{1}{\Gamma(n - \alpha)} \dfrac{d^{n}}{dx^n}\int_{a}^{x} \dfrac{f(y)}{(x - y)^{1  + \alpha -n} } \, dy    \qquad \forall x\in [a,b],\nonumber
                \end{align} 
                and the  $\alpha$ order right Riemann-Liouville fractional derivative of $f$ is defined by
                \begin{align}\label{right_RL_derivative}
                    {_{x}}{{D}}{^{\alpha}_{b}}f(x) :&= (-1)^n  \dfrac{d^n}{dx^n} \Bigl({_{x}}{I}{^{n-\alpha}_{b}} f(x) \Bigr)  \\
                    &= \dfrac{(-1)^n}{\Gamma(n - \alpha)} \dfrac{d^n}{dx^n} \int_{x}^{b}  \dfrac{f(y)}{(y - x)^{1 + \alpha -n} }  \,dy     \qquad \forall x\in [a,b].
                    \nonumber
                \end{align}
                ${_{a}}{{D}}{_{x}^{\alpha}} $ and ${_{x}}{{D}}{^{\alpha}_{b}}$ are called the left and right Riemann-Liouville fractional derivative(or differential) operators, respectively. 
            \end{definition}
    
            Another fractional derivative notion is the Caputo fractional derivative, which is widely used in initial value problems of factional order ODEs; particularly for fractional differentiation in time.
 
            \begin{definition}[cf. \cite{Samko}] \label{def2.3}
                Let $n-1 < \alpha <n$ and $f :[a,b] \rightarrow \R$.  The $\alpha$ order left Caputo fractional derivative of $f$ is defined by
                \begin{align}\label{left_Caputo}
                    {^{C}_{a}}{{D}}{_{x}^{\alpha}} f(x) := \dfrac{1}{\Gamma(n- \alpha)} \int_{a}^{x} \dfrac{f^{(n)}(y)}{(x- y)^{ 1+ \alpha -n}} \,dy \qquad \forall x\in [a,b],
                \end{align}
                and the $\alpha$ order right Caputo fractional derivative of $f$ is defined by
                \begin{align}\label{right_Caputo}
                    {^{C}_{x}}{{D}}{_{b}^{\alpha}} f(x) := \dfrac{(-1)^n}{\Gamma(n - \alpha)}  \int_{x}^{b} \dfrac{f^{(n)}(y)}{(y - x)^{1+\alpha -n}} \, dy \qquad\forall x\in [a,b].
                \end{align}
            \end{definition}

            \begin{remark}
                The definitions of the $\alpha$ order Caputo fractional derivatives require that $f^{(n)}$  exists almost everywhere. 
                The relationship between Riemann-Liouville and Caputo fractional derivatives is given by the following identities: 
                \begin{align}\label{RL_Caputo_l}
                    {^{C}_{a}}{{D}}{_{x}^{\alpha}} f(x) &:= {_{a}}{{D}}{_{x}^{\alpha}} f(x) -\sum_{k=0}^{n-1} \frac{f^{(k)}(a)}{\Gamma(k+1-\alpha)} (x-a)^{k-\alpha},\\
                    {^{C}_{x}}{{D}}{_{b}^{\alpha}} f(x) &:={_{x}}{{D}}{^{\alpha}_{b}} f(x) -\sum_{k=0}^{n-1}\frac{f^{(k)}(b)}{\Gamma(k+1-\alpha) } (b-x)^{k-\alpha}.
                    \label{RL_Caputo_r} 
                \end{align}
                It is easy to check that the following ``weak" definitions of the Caputo fractional derivatives in the case $0<\alpha<1$: 
                \begin{align}\label{weak_RL_Caputo_l} 
                    {^{C}_{a}}{{D}}{_{x}^{\alpha}} f(x) &:= {_{a}}{{D}}{_{x}^{\alpha}} \bigl[ f(x) -f(a) \bigr] =\frac{d}{dx} \Bigl[ {_{a}}{I}{^{1-\alpha}_{x}} \Bigl( f(x)-f(a) \Bigr) \Bigr],\\
                    {^{C}_{x}}{{D}}{_{b}^{\alpha}} f(x) &:={_{x}}{{D}}{^{\alpha}_{b}} \bigl[ f(x)-f(b) \bigr] =-\frac{d}{dx} \Bigl[{_{x}}{I}{^{1-\alpha}_{b}} \Bigl( f(x) -f(b) \Bigr) \Bigr],
                    \label{weak_RL_Caputo_r} 
                \end{align}
                which do not require the existence of $f'(x)$, instead, they require the existence of $f(a)$ and $f(b)$, respectively.
            \end{remark} 

        Unlike the Riemann-Liouville and Caputo derivatives which use an integral operator to induce a fractional order derivative, a natural question is if a fractional derivative can be defined as a limit of some difference quotient similar to the definition of the integer order derivative. Although there have been some attempts in this direction (cf. \cite{Khalil2014}), 
      we only recall the well-known Gr\"{u}nwald-Letnikov fractional derivatives as they are related to the Riemann-Liouville derivatives.
	
            \begin{definition} [cf. \cite{Samko}] 
               Let $0 < \alpha <1$ and $f:[a,b] \rightarrow \R$. The \textit{left} Gr\"{u}nwald-Letnikov fractional derivative of $f$ is defined by
                \begin{align*}
                    {^{GL}_{a}}{D}{^{\alpha}_{x}} f(x) : = \lim_{h \rightarrow 0^+} \dfrac{1}{h^{\alpha}} \sum_{k = 0 }^{[(x-a)/h]} \dfrac{(-1)^{k} \Gamma(1 + \alpha)}{\Gamma(k+1) \Gamma(\alpha - k +1)} f(x - kh) \qquad \forall x \in [a,b ]
                \end{align*}
                and the \textit{right} Gr\"{u}nwald-Letnikov fractional derivative of $f$ is defined by
                \begin{align*}
                    {^{GL}_{x}}{D}{^{\alpha}_{b}} f(x) : = \lim_{h \rightarrow 0^+} \dfrac{1}{h^{\alpha}} \sum_{k = 0 }^{[(b-x)/h]} \dfrac{(-1)^{k+1} \Gamma(1 + \alpha)}{\Gamma(k+1) \Gamma(\alpha - k +1)} f(x + kh) \qquad \forall x \in [a,b].
                \end{align*}
            \end{definition}

            Clearly, for the fractional derivative, the difference quotients are much more complicated. It can be shown \cite{Samko} that the Gr\"{u}nwald-Letnikov fractional derivative and the Riemann-Liouville derivative are equivalent for sufficient smooth functions.

        \subsubsection{\bf Definitions on an Infinite Interval}\label{sec-2.1.2} 
            The fractional integrals over unbounded intervals are defined in the same way; here we only consider the whole real line case, that is,  $(a,b)=(-\infty, \infty)$. 
            There are two different definitions of fractional order derivatives in the infinite interval case which were proved to be equivalent. The first three definitions are direct generalizations of Definitions \ref{def2.1}--\ref{def2.3}.

            \begin{definition}[cf. \cite{Samko}] \label{def2.4} 
                Let $\sigma > 0$ and $f:\R \rightarrow \R$. The $\sigma$ order left Liouville fractional integral of $f$ is defined by
                \begin{align*}
                    {}{I}{^{\sigma}_{x}} f(x) : =   \dfrac{1}{\Gamma(\sigma)} \int_{-\infty}^{x} \dfrac{f(y)}{(x - y)^{1 - \sigma}} \, dy \qquad \forall x \in \R
                \end{align*}
                and the $\sigma$ order right Liouville fractional integral of $f$ is defined by 
                \begin{align*}
                    {_{x}}{I}{^{\sigma}} f(x) : =  \dfrac{1}{\Gamma(\sigma)} \int_{x}^{\infty} \dfrac{f(y)}{(y-x)^{1 - \sigma}}\, dy \qquad \forall x \in \R.
                \end{align*}
            \end{definition}


            \begin{definition}[cf. \cite{Samko}] \label{def2.5} 
                Let $n-1 < \alpha < n$ and $f:\R \rightarrow \R$. The $\alpha$ order left Liouville fractional derivative of $f$ is defined by 
                \begin{align*}
                    {}{{D}}{_{x}^{\alpha}} f(x):= \dfrac{1}{\Gamma(n - \alpha)} \dfrac{d^{n}}{dx^n}\int_{-\infty}^{x} \dfrac{f(y)}{ (x - y)^{1+\alpha -n}} \,dy \qquad \forall x \in \R
                \end{align*}
                and the $\alpha$ order right Liouville fractional derivative of $f$ is defined by
                \begin{align*}
                    {_{x}}{{D}}{^{\alpha}}f(x) :=  \dfrac{(-1)^{n}}{\Gamma(n - \alpha)} \dfrac{d^n}{dx^n} \int_{x}^{\infty} \dfrac{f(y)}{(y - x)^{1+ \alpha -n}} \,dy \forall x \in \R.
                \end{align*}
            \end{definition}

            \begin{definition}[cf. \cite{Samko}] \label{def2.6} 
                Let $n-1 <\alpha< n $ and $f :\R \rightarrow \R$. The $\alpha$ order left Caputo fractional derivative of $f$ is defined by 
                \begin{align*}
                    {^{C}}{{D}}{_{x}^{\alpha}} f(x) &:= \dfrac{1}{\Gamma(n- \alpha)} \int_{- \infty}^{x} \dfrac{f^{(n)}(y)}{(x- y)^{1+ \alpha -n} } \,dy \qquad \forall x \in \R
                \end{align*}
                and the $\alpha$ order right Caputo fractional derivative of $f$ is defined by
                \begin{align*}
                    {^{C}_{x}}{{D}}{^{\alpha}} f(x) &:= \dfrac{(-1)^n}{\Gamma(n - \alpha)}  \int_{x}^{\infty} \dfrac{f^{(n)}(y)}{ (y - x)^{1+\alpha -n} } \, dy \qquad \forall x \in \R.
                \end{align*}
            \end{definition}

            It should also be noted that all integrals over the infinite domain are understood as standard improper integrals.

           	Similar to the finite interval case, we also can define the Gr\"unwald-Letnikov fractional derivatives for functions defined on the whole real line. In this case, notice that the sums are infinite sums in the above definition. 
 
            \begin{definition}[cf. \cite{Samko}] 
                Let $0 < \alpha <1$ and $f : \R \rightarrow \R$. The left Gr\"unwald-Letnikov fractional derivative of $f$ is defined by 
                \begin{align*}
                   {^{GL}}{D}{^{\alpha}_{x}} f(x) := \lim_{h \rightarrow 0^+} \dfrac{1}{h^{\alpha}} \sum_{k = 0}^{\infty} \dfrac{(-1)^{k}\Gamma(1+\alpha)}{\Gamma(k+1) \Gamma(\alpha - k +1)} f(x- kh) \qquad \forall x \in \R
               \end{align*}
               and the right Gr\"unwald-Letnikov fractional derivative of $f$ is defined by
               \begin{align*}
                    {^{GL}_{x}}{D}{^{\alpha}} f(x) := \lim_{h \rightarrow 0^+} \dfrac{1}{h^{\alpha}} \sum_{k = 0}^{\infty} \dfrac{(-1)^{k+1}\Gamma(1+\alpha)}{\Gamma(k+1) \Gamma(\alpha - k +1)} f(x+ kh) \qquad \forall x \in \R.
                \end{align*}
            \end{definition}

            Next, we recall another definition of fractional derivatives that are based on the Fourier transforms.  

            \begin{definition}[cf. \cite{Samko}] \label{def2.7}
	            Let $\alpha > 0$ and $f: \R \rightarrow \R$. The $\alpha$ order Fourier fractional derivative is defined by
	            \begin{align*}
	                {^{\mathcal{F}}}{{D}}{^{\alpha}}f(x) &:= \mathcal{F}^{-1} \left[ (i \xi)^{\alpha} \mathcal{F} [f](\xi)\right](x) \qquad \forall x \in \R
	            \end{align*}
                where $\mathcal{F}[\cdot]$ and $\mathcal{F}^{-1}[\cdot]$ denote respectively the Fourier transform and its inverse transform which are defined as follows: for any $x,\xi\in \R$
                \begin{align*}
                    \mathcal{F}[f](\xi) : = \int_{\R} e^{-i \xi x}f(x)\,dx,  \qquad \mathcal{F}^{-1}[f](x) := \int_{\R} e^{i\xi x}f(\xi)\,d\xi .
                \end{align*}
            \end{definition}

            \begin{remark}
	            The above Fourier fractional order derivative notion is based on the following well-known property of the Fourier transform:
                \[
                    \mathcal{F}[f^{(n)}](\xi)= (i \xi)^n \mathcal{F}[f](\xi), \qquad f^{(n)}(x)= \mathcal{F}^{-1}[ (i \xi)^n \mathcal{F}[f] ] (x)
                \]
                for any positive integer $n$.
            \end{remark}

    \subsection{Action on Smooth Functions with Compact Support}\label{sec-2.7}
        The action of the Riemann-Liouville integral and differential operators 
        on smooth functions with compact support is of special interest for our study 
        in this paper. The need for understanding these behaviors will become evident 
        in the later sections. 

        We now have a closer look at the support and the tail behavior of ${^{\pm}}{D}{^{\alpha}} \varphi$ for $\varphi \in C^{\infty}_{0}(\R)$ so 
        that $\mbox{supp}(\varphi)\subset (a,b)$. To that end, a direct computation yields 
        \begin{align}\label{eq2.20a}
            {^{-}}{{I}}{^{\sigma}_{x}} \varphi (x) = 
                \begin{cases} 
                    0 &\text{if } x \in (-\infty , a ),\\
                    {_{a}}{{I}}{^{\sigma}_{x}} \varphi(x) & \text{if } x \in [a,b],\\
                    L(x) & \text{if } x \in (b ,\infty), 
                \end{cases}
        \end{align}
        where
        \begin{align}\label{2.20b}
            L(x) = \dfrac{1}{\Gamma( \sigma)} \int_{a}^{b} \dfrac{\varphi(y)}{(x -y)^{1 - \sigma}} \, dy.
        \end{align}
        Taking the first derivative and letting $ \sigma=1-\alpha$ yields
        \begin{align} \label{LAction}
            {}{{D}}{^{\alpha}_{x}} \varphi(x) = 
                \begin{cases} 
                    0 &\text{if } x \in ( -\infty , a),\\
                    {_{a}}{{D}}{^{\alpha}_{x}} \varphi(x) & \text{if } x \in  [a,b],\\
                    L'(x) &\text{if } x \in  (b, \infty),
                \end{cases}
        \end{align}
   A similar formula can be shown for the right direction. The pollution function 
  will be denoted by $R(x)$ corresponding to $L(x)$.
          
        \begin{proposition} \label{Pollution}
            If $\varphi \in C^{\infty}_{0} (\R)$ with $\supp(\varphi) \subset (a,b)$, then $\supp({}{{D}}{^{\alpha}_{x}} \varphi)\subset (a, \infty)$ and $\supp({_{x}}{{D}}{^{\alpha}} \varphi)\subset (-\infty, b)$. 
        \end{proposition}

        \begin{remark}
    	    (a) Riemann-Liouville fractional differential (and integral) operators have a pollution effect on the support when acting on functions in $C^\infty_0 (\Ome)$. Left derivatives pollute the support to the right and right derivatives pollute the support to the left. 
    	    This pollution effect is a consequence of the nonlocal characteristics of fractional order differential and integral operators; in particular, the ``memory" effect.  
    	 
    	    (b) When $x\to \pm \infty$, the integrands in $L'(x)$ and $R'(x)$  
    	    are shrinking. Moreover, $\lim_{x \rightarrow \infty} |L'(x)| = 0$ (and $\lim_{x \rightarrow -\infty} |R'(x)| = 0$). 
        \end{remark}
       
        The next theorem states an integrability property of $D^\alpha \varphi$ for $\varphi \in C^{\infty}_{0} (\R)$. 
        
        \begin{theorem}
           Let $0<\alpha< 1$.  If $\varphi \in C^{\infty}_{0}(\Omega)$, then ${^{\pm}}{D}{^{\alpha}} \varphi \in L^{p}(\Omega)$ for each $1 \leq p \leq \infty$.
        \end{theorem}

        \begin{proof} 
        The proof comes from direct calculations that can be found in \cite{Feng_Sutton}.
        \end{proof}

    \begin{remark}
    	A special class of compactly supported smooth functions are those obtained through 
    	a mollification process (i.e., through a convolution with a 
    	 compactly supported mollifier). We refer the reader to 
    	 \cite{Feng_Sutton} for a detailed discussion. 
    \end{remark}

\subsection{Fundamental Theorem of Classical Fractional Calculus (FTcFC)}\label{sec-3.1}
    In this subsection we present an alternative understanding of 
	the classical fractional order integrals and derivatives. That is to interpret 
	the fractional differentiation as a by-product of the fractional integration 
	through the so-called {\em  Fundamental Theorem of Fractional Calculus}. 
	This new interpretation will play an important role in the development of 
	our weak fractional calculus theory to be given in the next section.
 
\subsubsection{\bf FTcFC on Finite Intervals $(a,b)\subset \R$}\label{sec-3.1.1}
We begin this subsection by recalling the following properties of the fractional operators 
${^{\pm}}{I}{^{\alpha} }$ and $  {^{\pm}}{D}{^{\alpha} }$.
        
    \begin{lemma} [cf. \cite{Samko}]\label{lemma3.1}
        Let $0 < \alpha <1$. The  following properties hold:
        \begin{itemize}
        \item[(a)]  ${_{a}}{D}{^{\alpha}_{x}} \kappa^{\alpha}_{-}(x) \equiv 0$ and ${_{x}}{D}{^{\alpha}_{b}} \kappa^{\alpha}_{+}(x) \equiv 0$ where
        \begin{align}\label{kernel_functions}
            \kappa^{\alpha}_{-}(x) := (x-a)^{\alpha -1}, \qquad \kappa^{\alpha}_{+}(x) := (b-x)^{\alpha-1} .
        \end{align}
            \item[(b)] ${_{a}}{D}{^{\alpha}_{x}}{_{a}}{I}{^{\alpha}_{x}} f(x) = f(x)$ and ${_{x}}{D}{^{\alpha}_{b}}{_{x}}{I}{^{\alpha}_{b}}f(x) = f(x)$ for any $f \in L^{1}_{loc}((a,b))$.
            \item[(c)] If ${_{a}}{I}{^{1 - \alpha}_{x}}f \in AC([a,b])$, then 
            \begin{equation}\label{FTFC_1}
            f(x) = c^{1-\alpha}_{-} \kappa^{\alpha}_{-}(x) + {_{a}}{I}{^{\alpha}_{x}}{_{a}}{D}{^{\alpha}_{x}}f(x),
            \end{equation} 
            and if ${_{x}}{I}{^{1 - \alpha}_{b}}f \in AC([a,b])$, then 
            \begin{equation}\label{FTFC_2}
            f(x) = c^{1-\alpha}_{+} \kappa^{\alpha}_{+}(x) + {_{x}}{I}{^{\alpha}_{b}} {_{x}}{D}{^{\alpha}_{b}}f(x),
            \end{equation}
            where 
           \begin{align} \label{FTFC_3}
            c_{-}^{\sigma}  := \frac{  {_{a}}{I}{^{\sigma}_{x}} f(a) }{\Gamma(\sigma)},\qquad
            c_{+}^{\sigma} := \frac{   {_{x}}{I}{^{\sigma}_{b}} f(b) }{\Gamma(\sigma)}. 
            \end{align}
        \end{itemize}
    
    \end{lemma}

On noting the fact that ${^{\pm}}{D}{^{\alpha}}F(x) = f(x)$ implies that ${^{\pm}}{I}{^{1-\alpha}}F \in AC([a,b])$, then the above lemma immediately infers the 
following theorem. 

    \begin{theorem} \label{FTFC} 
        Let $0<\alpha <1$, $f, F \in L^{1}((a,b))$. 
        Then ${^{\pm}}{D}{^{\alpha}}F(x) = f(x)$  on $(a,b)$ if and only if 
        \begin{equation}\label{FTFC_4} 
        F(x) = c_{\pm}^{1-\alpha} \kappa^{\alpha}_{\pm}(x) + {^{\pm}}{I}{^{\alpha}}f(x).
        \end{equation}
  
    \end{theorem}

\begin{remark}     
 The analogue of Theorem \ref{FTFC} in the integer order calculus is the well-known 
 {\em Fundamental Theorem of Calculus (or Newton-Leibniz Theorem)} which says that $F'(x):=\frac{dF}{dx}(x)=f(x)$ if 
 and only if 
 \[
 F(x)= F(a) + \int_a^x f(y)\, dy =  F(a) + \int_a^b H(x-y) f(y)\, dy\qquad \forall x\in [a,b],
 \]
 where the kernel function $\kappa(x,y)= H(x-y)$, the Heaviside function. Since 
 the kernel space of the derivative $\frac{d}{dx}$ operator is $\R$, this is why the first 
 term  on the right-hand side must be a constant, because it must belong to the kernel 
 space of $\frac{d}{dx}$. 
 Due to the above analogue, we shall call Theorem \ref{FTFC} {\em Fundamental Theorem of 
 	Classical Fractional Calculus} on finite intervals in the rest of this paper.
\end{remark}
    
In fact, given integral operators ${^{\pm}}{I}{^{\alpha}}$,   \eqref{FTFC_4} 
can be used to define the corresponding Riemann-Liouville derivatives as follows.

  \begin{definition} \label{FTFC_def} 
	Let $0<\alpha <1$ and $f , F\in L^{1}((a,b))$. Then $f$ is called  the $\alpha$ order left/right  
	Riemann-Liouville fractional derivative of $F$, and write ${^{\pm}}{D}{^{\alpha}}F(x) = f(x)$, 
	if \eqref{FTFC_4} holds,
	\begin{equation} \label{FTFC_5} 
	F(x) = c_{\pm}^{1- \alpha} \kappa^{\pm}(x) + {^{\pm}}{I}{^{\alpha}}f(x).
	\end{equation}

\end{definition} 

It is easy to check that the $\alpha$ order fractional derivative of $f$, if it exists, 
is uniquely defined.  In light of Theorem \ref{FTFC},  we see that the original definition 
and the above definition are equivalent. 
In this paper we emphasize the above FTFC approach of using a given integral 
operator (i.e.,  its kernel function is given) to define the corresponding derivative 
notion by the FTFC identity.  There are many benefits/advantages of this approach. 
It is systematic (not ad hoc) and quite general, because it is done in the same way for 
any given integral operator (see the definition below). The FTFC is built into the 
definition; we regard that having such a FTFC is essential for any fractional calculus 
theory.  

We now give the alluded definition of fractional derivatives for general kernels (and their associated integral operators). 

\begin{definition}
Given any kernel function $\tau\in L^1((a,b)\times (a,b))$, let 
$I_\tau$ denote the subordinate (Riemann or Lebesgue) integral operator, namely,
\begin{align}
I_{\tau} f(x):= \int_{a}^b \tau(x,y) f(y)\, dy    \qquad \forall x\in [a,b].
\end{align}
Let $f, F\in L^1(\Ome)$, then $f$ is called the fractional/nonlocal derivative of $F$, and written
$D_{\tau} F = f$, there exists some $c\in [a,b]$ such that 
       \begin{align} \label{general_def_a}
           F(x) &=   C_{F,c} \tau(x,c) +  I_{\tau}  f(x) \qquad \forall x\in [a,b] 
       \end{align}
       for some constant $C_{F,c}$ depending on both $F$ and $c$.
\end{definition}

\subsubsection{\bf FTcFC on the Infinite Interval $\R$}\label{sec-3.1.2}
The case for a FTcFC on the entire line is quite different, but simpler 
because of the decay properties of kernel functions $\kappa^{\alpha}_{\pm}$ when $|x|\to \infty$. Similarly, we start by recalling the following properties 
of the fractional operators ${^{\pm}}{I}{^{\alpha} }$ and ${^{\pm}}{D}{^{\alpha} }$.  

        \begin{lemma}[cf. \cite{Samko}] \label{lem3.4}
        	Let $0 < \alpha < 1$. The following properties hold:
        \begin{itemize}
            \item[(a)] $D^{\alpha}_{x} I^{\alpha}_{x} f(x) = f(x)$ and ${_{x}}{D}{^{\alpha}} {_{x}}{I}{^{\alpha}}f(x) = f(x)$ for any $f \in L^{1}(\R)$,
            \item[(b)] $I^{\alpha}_{x} D^{\alpha}_{x} f(x) = f(x)$ and ${_{x}}{I}{^{\alpha}} {_{x}}{D}{^{\alpha}} f(x) =  f(x)$ for any $I^{1-\alpha}f \in AC(\R)$ so that $f(x) \rightarrow 0$ as $|x| \rightarrow \infty$.
        \end{itemize} 
    \end{lemma}
    
    We then have
    
    \begin{theorem}\label{FTFCa}
        Let $0 < \alpha < 1$, and $f, F \in L^{1}(\R)$. If
        \begin{equation}\label{FTFC_1a}
               F(x)= {^{\pm}}{I}{^{\alpha}} f(x),
        \end{equation} 
        then ${^{\pm}}{D}{^{\alpha}} F(x) = f(x)$. The converse is also true if 
        $F(x)\to 0$ as $|x|\to \infty$ is required.
    \end{theorem}

For the same reason as given in Subsection \ref{sec-3.1.1}, we shall call 
Theorem \ref{FTFCa} {\em the Fundamental Theorem of Classical Fractional Calculus on $\R$}
in the rest of this paper. 

Similarly, we also introduce the following definition.

  \begin{definition} \label{FTFC_defa} 
	Let $0<\alpha <1$ and $f , F\in L^{1}(\R)$. Then $f$ is called  the $\alpha$ order left/right  
	Riemann-Liouville fractional derivative of $F$ on $\R$, and write ${^{\pm}}{D}{^{\alpha}}F(x) 
	= f(x)$ (abusing the notation), if \eqref{FTFC_1a} holds.  
\end{definition} 

It is easy to show that ${^{\pm}}{D}{^{\alpha}}F$ is well defined and it coincides with 
the original definitions of Riemann-Liouville derivatives on $\R$. 
This FTcFC interpretation of fractional derivatives will be emphasized in this paper.

\section{A Weak Fractional Differential Calculus Theory}\label{sec-4}
We saw in the previous section that the classical fractional calculus 
theory has several difficulties arising from the change to non-integer order 
integration and differentiation. Unlike the well formulated and understood 
integer order calculus, the basic notion of fractional derivatives is 
domain-dependent and has several different (and nonequivalent) 
definitions; familiar calculus rules do not hold or become fairly complicated 
and restricted; fractionally differentiable functions are difficult to characterize;  
there is no local characterization of non-local fractional integral and 
derivative operators; more importantly, although the Riemann integration 
can be generalized to the Lebesgue integration in the definitions of all 
fractional integrals, unlike the integer order case, there is no weak 
fractional derivative concept/theory, 
which in turn has caused some difficulties and confusions for studying/interpreting 
fractional order differential equations.

The primary goal of this section (and this paper) is to develop a weak fractional 
differential calculus theory, which is parallel to the integer order weak derivative theory 
(cf. \cite{Adams, Brezis, Evans}). The anticipated weak fractional theory 
lays down the ground work for developing a new fractional Sobolev space theory
in a companion paper \cite{Feng_Sutton1a}. Together they will provide a solid theoretical 
foundation and  pave the way for a systematic and thorough study of initial value, 
boundary value and initial-boundary value problems for fractional order differential 
equations and fractional calculus of variations problems as well as 
their numerical solutions in the subsequent works \cite{Feng_Sutton2,Feng_Sutton3}.


In this section, unless it is stated otherwise, all integrals are understood in the Lebesgue sense. We use ${^{-}}{D}{^{\alpha}}$ and ${^{+}}{D}{^{\alpha}}$ to denote respectively any left and right $\alpha$ order classical derivative introduced in Section \ref{sec-2}. ${^{\pm}}{D}{^{\alpha}}$ denotes either ${^{-}}{D}{^{\alpha}}$ or ${^{+}}{D}{^{\alpha}}$. $\Omega$ denotes either a finite interval $(a,b)$ or the whole real line $\R$.  In the case $\Ome=(a,b)$, for any $\varphi \in C^{\infty}_{0}(\Omega)$, $\tilde{\varphi}$ is used to denote the zero extension of $\varphi$ to $\R$.

        
\subsection{Definitions of Weak Fractional Derivatives}\label{sec-4.1}
%
Like in the integer order case, the idea of defining {\em weak} fractional derivative ${^{\pm}}{ \mathcal{D}}{^{\alpha}} u$ of a function $u$ is to specify its action on any smooth compactly supported function $\varphi \in C^{\infty}_{0}(\Omega)$, 
instead of knowing its pointwise values as done in the classical fractional derivative definitions.

    \begin{definition}\label{RWFD}
        For $\alpha> 0$, let $[\alpha]$ denote the integer part of $\alpha$. For $u \in L^{1}(\Omega)$, 
       \begin{itemize} 
       \item[{\rm (i)}] a function $v \in L_{loc}^{1} (\Omega)$ is called the left weak fractional derivative of $u$ if 
        \begin{align*}
            \int_{\Omega} v(x) \varphi(x) \,dx = (-1)^{[\alpha]} \int_{\Omega} u(x) {^{+}}{D}{^{\alpha}} \tilde{\varphi}(x) \, dx
             \qquad \forall \varphi \in C_{0}^{\infty} (\Omega),
        \end{align*}
        we write ${^{-}}{ \mathcal{D}}{^{\alpha}} u:=v$; 
     \item[{\rm (ii)}] a function $w\in L_{loc}^{1} (\Omega)$ is called the right weak fractional derivative of $u$ if 
      \begin{align*}
       \int_{\Omega} w(x) \varphi(x) \,dx = (-1)^{[\alpha]} \int_{\Omega} u(x) {^{-}}{D}{}^{\alpha} \tilde{\varphi}(x) \,dx
       \qquad \forall \varphi \in C_{0}^{\infty} (\Omega), 
      \end{align*}
       and we write ${^{+}}{\mathcal{D}}{^{\alpha}} u:=w$. 
      \end{itemize}
   \end{definition}

The next proposition shows that weak fractional derivatives are well-defined. 

\begin{proposition}
	Let $u \in L^1 (\Omega)$. Then a weak fractional derivative of $u$, if it exists, is uniquely defined.
\end{proposition}

\begin{proof}
	Let $v_1, v_2 \in L^1_{loc} ( \Omega)$ be two left (resp. right) weak fractional derivatives of $u$, then
	\begin{align*}
	\int_{\Omega} v_1(x) \varphi(x) \, dx = (-1)^{[\alpha]} \int_{\Omega} u(x) {^{\pm}}{D}{^{\alpha}} 
	\tilde{\varphi}(x) \, dx 
	=   \int_{\Omega} v_2(x) \varphi(x) \, dx  \quad\forall \varphi \in C_{0}^{\infty} (\Omega).
	\end{align*}
	Thus, 
	\begin{align*}
	0 = \int_{\Omega} \big(v_1(x) - v_2(x) \big) \varphi(x) \, dx \qquad\forall \varphi \in C^{\infty}_{0} (\Omega).
	\end{align*}
	Therefore, $v_1 = v_2$ almost everywhere. The proof is complete.
\end{proof}

\smallskip
  A few remarks are given below to help understand the above definition. 
  
    \begin{remark}
       (a) The introduction of $\tilde{\varphi}$ in the definitions makes the weak fractional derivatives intrinsic 
       in the sense that ${^{\pm}}{D}{}^{\alpha} \tilde{\varphi}$ is independent of the choice of ${^{\pm}}{D}{}^{\alpha}$, because 
       ${ }{D}{^{\alpha}_{x}}\tilde{\varphi}={_{a}}{D}{^{\alpha}_{x}}\tilde{\varphi}= {^{\mathcal{F}}}{D}{^{\alpha}}\tilde{\varphi}$ and  
       ${_{x}}{D}{^{\alpha} }\tilde{\varphi}={_{x}}{D}{^{\alpha}_{b}}\tilde{\varphi}= {^{\mathcal{F}}}{D}{^{\alpha}}\tilde{\varphi}$.
       
       (b) The constant $(-1)^{[\alpha]}$ helps guarantee consistency with the integer order case. 
       
      (c) Integration by parts is built into the definitions.
       
      (d) The reason to require $u \in L^{1}(\Omega)$ is because ${^{\pm}}{D}{^{\alpha}} \tilde{\varphi} \in L^{\infty}(\R)$  
       is not compactly supported. When $\alpha \in \N$, this condition can be relaxed to $L^{1}_{loc}(\Omega)$. In fact, the restriction $u \in L^{1}(\Omega)$ can be relaxed to the weighted $L^1$ space 
       $u\in L^1(\Omega, \rho)$ with the weight $\rho=L'$ or $\rho=R'$.
       
       (e)  As expected, weak fractional derivatives are {\em domain-dependent}. 
       	However, unlike 
       the classical fractional derivatives, whose domain dependence is explicitly shown in the limits of the integrals
       involved, the domain dependence of weak fractional derivatives is implicitly introduced by using 
       domain-dependent test functions $\varphi\in C^\infty_0(\Omega)$.  
       
      (f)  The above definitions can be easily extended to non-interval domains or subdomains of $\Omega$.  Indeed, given  a bounded set $E\subset \R$,  the 
      only changes which need  to be made in the definitions are to replace $\Omega$ by $E$ and $\varphi\in C^\infty_0(\Omega)$  by  $\varphi\in C^\infty_0((a^*, b^*))$  where $(a^*,b^*) =\cap\{(c,d):\, E\subset (c,d)\}$,
       the smallest interval which contains $E$.
    
       (g) Extensions of the above definitions to distributions will be given in Section \ref{sec-6}.
      \end{remark}

\smallskip
 The following result is trivial and expected, see \cite[Theorem 2.5]{Feng_Sutton}
	for a proof. 

 \begin{proposition}\label{Weak=RL}
	Let $u$ be Riemann-Liouville differentiable such that ${^{\pm}}{D}{^{\alpha}}u \in L^{1}_{loc}(\Omega)$. Then ${^{\pm}}{\mathcal{D}}{^{\alpha}} u = {^{\pm}}{D}{^{\alpha}} u$ almost everywhere.
\end{proposition}

The next result shows the consistency with integer order weak derivatives.

\begin{proposition}
    Let $n - 1 < \alpha <n$. The $\alpha$ order weak fractional derivative converges to the $n^{th}$ order weak derivative almost everywhere as $\alpha \rightarrow n$. 
\end{proposition}

\begin{proof}
    Consider the case when $n=1$; the others follow similarly. In order to prove that ${^{\pm}}{\mathcal{D}}{^{\alpha}} u \rightarrow \mathcal{D}u$ almost everywhere as $\alpha \rightarrow 1$, we see that
    \begin{align*}
        0 
        &= \int_{\Omega} u \varphi'\,dx + \lim_{\alpha \rightarrow 1} (-1)^{[\alpha]} \int_{\Omega} u {^{\mp}}{D}{^{\alpha}} \varphi\,dx \\ 
        &= \lim_{\alpha \rightarrow 1}  \int_{\Omega} {^{\pm}}{\mathcal{D}}{^{\alpha}} u \varphi\,dx - \int_{\Omega} \mathcal{D}u \cdot \varphi\,dx \\
        &= \lim_{\alpha \rightarrow 1} \int_{\Omega} (  {^{\pm}}{\mathcal{D}}{^{\alpha}} u  - \mathcal{D}u) \varphi\,dx,
    \end{align*}
    which follows by the consistency of classical derivatives on functions $\varphi \in C^{\infty}_{0}(\Omega)$.
\end{proof}

    \subsection{Relationships with Other Derivative Notions}\label{sec-4.2}
    Although the notion of a weak fractional derivative is analogous to the integer order weak derivative and hence is deserving of the name in this sense, we provide simple examples to illustrate the following points.
    \begin{itemize}
        \item [(a)] The notion of a weak fractional derivative is a unifying concept of fractional differentiation with respect to the derivatives defined in Section \ref{sec-2}.
        \item[(b)]  Weak fractional derivatives can exist for functions whose classical fractional
        derivatives do not exist. 
        \item[(c)] Functions that do not have first order weak derivatives may have weak fractional derivatives. 
    \end{itemize}

    First, we give a simple example to demonstrate that the weak fractional derivative is a unifying concept (item (a) above). To illustrate this point, we consider 
    $\Omega = \R$, $0 < \alpha <1$, and $c \in \R \setminus\{0\}$, let $u(x) \equiv c$. Trivially, 
    \begin{align*}
    {^{C}}{D}{^{\alpha}_{x}} u(x) = \dfrac{1}{\Gamma(1-\alpha)} \int_{-\infty}^{x} \dfrac{u'(y)}{ (x-y)^{-\alpha}}\,dy  
    =  0, 
    \end{align*}
    hence, the Caputo derivative is identically zero.  However,  
    \begin{align*}
    {}{D}{^{\alpha}_{x}} u(x) &= \dfrac{1}{\Gamma(1-\alpha)} \dfrac{d}{dx}\int_{-\infty}^{x} \dfrac{c}{(x-y)^{\alpha}} \,dy = \dfrac{c}{\Gamma(1-\alpha)}\dfrac{d}{dx} \left( \dfrac{(x-y)^{1-\alpha}}{1-\alpha} \bigg|_{y=x}^{y=-\infty} \right),
    \end{align*}
    which does not exist as a function because the singular integral diverges, hence, 
    the Riemann-Liouville fractional derivative does not exist on $\R$. Thus in the classical case, the choice of fractional derivative definition becomes essential. 
    
   We now compute the weak fractional derivative of $u$ below. For any $\varphi \in C^{\infty}_{0}(\R)$,
    \begin{align*}
        \int_{\R} c {^{+}}{D}{^{\alpha}} \varphi(x)\,dx &= \int_{\R} c \dfrac{d}{dx} {}{I}{^{1-\alpha}_{x}} \varphi(x) \,dx
        = c \Bigl[ {}{I}{^{1-\alpha}_{x}} \varphi(x) \Bigr] \bigg|_{-\infty}^{\infty} = 0.
    \end{align*}
    Therefore, the weak derivative exists and is equal to zero, which coincides 
    with the Caputo derivative. Here we see that by forcing the integration by parts 
    formula to hold, the definition automatically selects the appropriate fractional derivative.
    
    What if $\Omega = (a,b)$ is finite? In this case, we know that $0={^{C}_{a}}{D}{^{\alpha}_{x}} c \neq {_{a}}{D}{^{\alpha}_{x}} c 
    = c \Gamma(1-\alpha)^{-1} (x-a)^{-\alpha}$. 
    A simple calculation yields that 
    \begin{align*}
        \int_{a}^{b} c {^{+}}{D}{^{\alpha}}\varphi(x)\, dx &= c\int_{a}^{b}  \dfrac{d}{dx} {_{x}}{I}{^{\alpha}_{b}} \varphi(x)\,dx = c \Bigl[ {_{x}}{I}{^{\alpha}_{b}} \varphi(x) \Bigr] \bigg|_{a}^{b} = c\, {_{a}}{I}{^{\alpha}_{b}}\varphi(b)
    \end{align*}
     holds for all $\varphi \in C^{\infty}_{0}((a,b))$, which shows that the Caputo 
     derivative of constant $c$ (that is zero) can not satisfy the integration by parts 
     formula. Hence, 
     $ {^{C}_{a}}{{D}}{_{x}^{\alpha}} c\neq {^{-}}{\mathcal{D}}{^{\alpha}} c$.  
     However, a direct computation shows that 
    \begin{align*}
        \int_{a}^{b} c {^{+}}{D}{^{\alpha}}\varphi\,dx  = \int_{a}^{b} \varphi  {_{a}}{D}{^{\alpha}
        _{x}} c\,dx  \qquad \forall \varphi \in C^{\infty}_{0}((a,b)).
    \end{align*}
    Hence, ${^{-}}{\mathcal{D}}{^{\alpha}} c = {_{a}}{D}{^{\alpha}_x} c$. Again, we see that 
    the built-in feature of an integration by parts formula effectively selects an appropriate fractional derivative.

    Next, we illustrate that the notion of weak fractional derivatives is truly a generalization 
    of the notion of classical fractional derivatives by showing that there are functions whose
    weak derivatives exist, but classical fractional (Rieamann-Liouville) derivatives do not. 
    Moreover, we give a characterization of functions that are weakly differentiable, which parallels the characterization for first order weakly differentiable functions.
   In lieu of concrete examples, we demonstrate that there is a procedural way to produce 
   functions that are not Riemann-Liouville differentiable, but are weakly differentiable. 
   Notice that for $u \in L^{1}(\Omega)$ and $\varphi \in C^{\infty}_{0}(\Omega)$, there holds
    \begin{align*}
        \int_{\Omega} u {^{\mp}}{D}{^{\alpha}} \varphi\,dx &= \int_{\Omega} u {^{\mp}}{I}{^{1-\alpha}} \varphi'\,dx = \int_{\Omega} {^{\pm}}{I}{^{1-\alpha}} u \varphi '\,dx.
    \end{align*}
    In order to perform an integration by parts on the right side, we need that ${^{\pm}}{I}{^{1-\alpha}}u\in W^{1,1}(\Omega)$ (or at least absolutely continuous).
    In that case, the function $u$ then has a weak fractional derivative. On the other hand,
     we want the function 
    $u$ not to be Riemann-Liouville differentiable, which requires that ${^{\pm}}{I}{^{1-\alpha}} u \not \in C^{1}(\Omega)$. Since ${^{\pm}}{I}{^{1-\alpha}}u \in W^{1,1}(\Omega)$ does not imply ${^{\pm}}{I}{^{1-\alpha}}u \in C^{1}(\Omega) $, then we want to find $u \in L^{1}(\Omega)$ so that ${^{\pm}}{I}{^{1-\alpha}} u = f$ 
    for a given function $f \in W^{1,1}(\Omega)$, but $f \not\in C^{1}(\Omega)$. There are many such $f$ functions, the best known example perhaps is $f(x) = |x|$. 
 
    It follows from Lemma \ref{lemma3.1} that we obtain the desired 
    examples by taking $u  = {^{\pm}}{D}{^{1-\alpha}} f$ 
    for any $f \in \bigl\{ W^{1,1}(\Omega); f \not\in C^{1}(\Omega) \mbox{ and }
    {^{\pm}}{D}{^{1-\alpha}}f \mbox{ exists} \bigr\}$. By the characterization of functions in $W^{1,1}(\Omega)$, 
    we conclude that
	  \textit{$u$ is weakly fractional differentiable with ${^{\pm}}{\mathcal{D}}{^{\alpha}} u \in L^{1}(\Omega)$ if and only ${^{\pm}}{I}{^{1-\alpha}} u$ is absolutely continuous.
    } 
   
   \begin{remark} The above procedure can be relaxed to characterize all weakly fractional differentiable functions by requiring $f$ to be only first order weakly differentiable; rather than $f \in W^{1,1}(\Omega)$. However, the above procedure does produce a rich (and nearly complete) validation of item (b) above. 
   \end{remark}
    
    Finally, we compare the weak fractional derivative to the integer order weak derivative; in particular, we demonstrate that the notion of weak fractional derivative is indeed consistent with, and extends, the notion of integer order weak derivatives by identifying a class of functions so that their weak fractional derivatives exist, but their integer order weak derivatives do not. 

    To that end, consider $\Omega = (-1,1)$ and $\lambda , \mu \in \R$ so that $\lambda \neq \mu$, 
    then define
    \begin{align*}
        u(x) : = \begin{cases}
            \lambda &\text{if } -1<x <0, \\ 
            \mu &\text{if } 0<x <1;
        \end{cases}
    \end{align*}
    a genuine step function. Let $\mathcal{D}$ denote the first order weak derivative operator. 
    Obviously, $\mathcal{D}u$ does not exist (cf. \cite{Brezis}) because $u \not\in C((-1,1))$; such a function has only a distributional derivative. 
    However,  a direct calculation shows that  
    \begin{align*}
        \int_{-1}^{1} u  {^{\mp}}{D}{^{\alpha}} \varphi \,dx = \int_{-1}^{1}  \varphi {^{\pm}}{ \mathcal{D}}{^{\alpha}}u \,dx \qquad \forall \varphi\in C^\infty_0((-1,1))
    \end{align*}
    holds, where 
    \begin{align*}
        {^{-}}{ \mathcal{D} }{^{\alpha}} u(x) := \begin{cases}
         \dfrac{1}{\Gamma(1-\alpha)} \dfrac{\lambda}{(x+1)^{\alpha}}&\text{if } x \in (-1,0], \\ 
         \dfrac{1}{\Gamma(1-\alpha)} \left( \dfrac{\lambda}{(x+1)^{\alpha}} - \dfrac{\lambda}{x^{\alpha}} + \dfrac{\mu}{x^{\alpha}} \right) &\text{if } x \in (0,1).
        \end{cases} 
    \end{align*}
    A similar formula also holds for ${^{+}}{ \mathcal{D} }{^{\alpha}}$. Note that the weak derivative is locally integrable. In fact, since $0 < \alpha <1$, it is globally integrable; an observation that is foundational to density properties  in the fractional Sobolev spaces introduced in \cite{Feng_Sutton1a}.
    Thus, we have shown that all step functions are weakly fractional differentiable, but are not weakly differentiable to any integer order. In fact, it can be shown that the same conclusion also holds for all piecewise smooth, but globally discontinuous functions. Simple exams are given in \cite{Brezis, Evans}.
    
    \subsection{Approximation and Characterization of Weak Fractional Derivatives}\label{sec-4.3}
    In this subsection we present a characterization for weak fractional derivatives so that they can be approached/understood from a different, but equivalent point of view. Like in the integer order case, 
    we prove that weakly fractional differentiable functions can be approximated by 
    smooth functions.  Unless it is stated otherwise, we assume $0<\alpha< 1$ in this subsection. 
    
    \subsubsection{\bf The Finite Interval Case}\label{sec-4.3.1}
     
     We first consider the case when $\Omega:=(a,b) \subset \R$ is a finite interval. Let 
    $\eps > 0$,  define the $\eps$- interior of $\Ome$ as $\Omega_{\eps} : = \{x \in \Omega \, : \mbox{dist}(x,\partial \Omega) >\eps\}.$

        \begin{lemma}\label{WDMollifier}
           Suppose ${^{\pm}}{\mathcal{D}}{^{\alpha}} u \in L_{loc}^{1} (\Omega)$ exists. Then
            \begin{align}\label{WeakMollifier}
                {^{\pm}}{\mathcal{D}}{^{\alpha}} \tilde{u}^{\eps} = \eta_{\eps} * {^{\pm}}{\mathcal{D}}{^{\alpha}}u \qquad \mbox{a.e. in } \Omega_{\eps}
            \end{align}
 where $\eta_{\eps}$ denotes the standard mollifier and $ \tilde{u}^{\eps} $ stands for the 
 mollification of $\tilde{u}$.
        \end{lemma}
        
      We omit the proof to save space and refer the reader to \cite[Lemma 3.3]{Feng_Sutton} for 
      details.

        The next theorem gives a characterization of fractional order weak derivatives. 
        
        \begin{theorem}
            Let $u \in L^{1}(\Omega)$. Then $v = {^{\pm}}{\mathcal{D}}{^{\alpha}} u$ in  $L^{1}_{loc} (\Omega)$ if and only if there exists a sequence $\left\{u_j \right\}_{j=1}^{\infty} \subset C^{\infty} (\Omega)$ such that $u_j \rightarrow u$ in $L^{1}(\Omega)$ and ${^{\pm}}{\mathcal{D}}{^{\alpha}} u_j \rightarrow v$ in $L^{1}_{loc}(\Omega)$ as $j \rightarrow \infty$. 
        \end{theorem}

        \begin{proof}
            Let $u \in L^{1}(\Omega)$ and $u^{\eps}$ denote its mollification. 

            {\em Step 1:}  Suppose that $v = {^{\pm}}{\mathcal{D}}{^{\alpha}} u \in L^{1}_{loc}(\Omega)$.  Let $\tilde{u}^{\eps}$ denote the  mollification  of $\tilde{u}$. By the properties of mollification, $\tilde{u}^{\eps} \rightarrow u$ in $L^{1}(\Omega)$ as $\eps\to 0$. From lemma, we have  ${^{\pm}}{\mathcal{D}}{^{\alpha}} \tilde{u}^{\eps} = \eta_{\eps} * {^{\pm}}{\mathcal{D}}{^{\alpha}}u 
             \rightarrow {^{\pm}}{\mathcal{D}}{^{\alpha}} u$ in $L^{1}_{loc}(\Omega)$ as 
            $\eps\to 0$. Hence, $\{\tilde{u}^{\eps} \}$ is a desired sequence. 
            
            \smallskip
            {\em Step 2:}  Suppose that $\left\{u_{j} \right\}_{j=1}^{\infty} \subset C^{\infty}(\Omega)$ and $u_j \rightarrow u$ in $L^{1}(\Omega)$ and ${^{\pm}}{\mathcal{D}}{^{\alpha}} u_{j} \rightarrow v$ in $L^{1}_{loc} ( \Omega)$. Then for any $\varphi \in C^{\infty}_{0}(\Omega)$
            \begin{align*}
                \left|\int_{\Omega} (u -u_j)(x) {^{\mp}}{D}{^{\alpha}} \varphi(x) \,dx \right| &\leq M\|u - u_j \|_{L^{1}(\Omega)}\to 0 \qquad \mbox{as } j\to \infty,\\
                \left|\int_{\Omega} \left({^{\pm}}{\mathcal{D}}{^{\alpha}} u_j - v\right)(x) \varphi(x) \,dx \right| &= \left| \int_{K} \left({^{\pm}}{\mathcal{D}}{^{\alpha}} u_j - v\right)(x) \varphi(x) \,dx \right|\\
                &\leq M \left\|{^{\pm}}{\mathcal{D}}{^{\alpha}} u_j - v \right\|_{L^{1}(K)} \rightarrow 0
                \quad\mbox{as } j\to \infty,
            \end{align*}
            because $K:= \supp(\varphi)$ is compact.  It follows from the definition of weak fractional derivatives that 
            \begin{align*}
                (-1)^{[\alpha]} \int_{\Omega} u(x) {^{\mp}}{D}{^{\alpha}} \varphi (x) \, dx
                &= (-1)^{[\alpha]} \lim_{j \rightarrow \infty} \int_{\Omega} u_j(x) {^{\mp}}{D}{^{\alpha}} \varphi (x) \, dx \\ 
                &= \lim_{j \rightarrow \infty} \int_{\Omega} {^{\pm}}{\mathcal{D}}{^{\alpha}} u_{j}(x) \varphi(x) \, dx 
                = \int_{\Omega} v(x) \varphi(x) \, dx.
            \end{align*}
            By the uniqueness of the weak fractional derivative, we conclude that $v = {^{\pm}}{\mathcal{D}}{^{\alpha}} u$ almost everywhere.  The proof is complete. 
        \end{proof}
        
       \begin{corollary}
            Let $u \in L^{p}(\Omega)$ for $1 \leq p < \infty$. Then $v = {^{\pm}}{\mathcal{D}}{^{\alpha}} u$ in  $L^{q}_{loc} (\Omega)$ for $1\leq q <\infty$ if and only if there exists a sequence $\left\{u_j \right\}_{j=1}^{\infty} \subset C^{\infty} (\Omega)$ such that $u_j \rightarrow u$ in $L^{p}(\Omega)$ and ${^{\pm}}{\mathcal{D}}{^{\alpha}} u_j \rightarrow v$ in $L^{q}_{loc}(\Omega)$ as $j \rightarrow \infty$.
           
        \end{corollary}
      
        \begin{remark}
            The conclusion of the above corollary still holds if $L^q_{loc}(\Omega)$ is replaced by $L^q(\Omega)$
            in the statement.
        \end{remark}

    \subsubsection{\bf The Infinite Domain Case}\label{sec-4.3.2}
    We now consider the case $\Omega = \R$. It turns out this case is significantly different
    from the finite interval case. In particular, it requires the construction of a compactly supported 
    approximation sequence for each fractionally differentiable function, which turns out is quite complicated. 
    
    First, we establish the following analogue of Lemma \ref{WDMollifier}. We refer the reader  
    to \cite[Lemma 3.3]{Feng_Sutton} for its proof. 
    
        \begin{lemma}\label{WDMollifierR}
        Suppose ${^{\pm}}{\mathcal{D}}{^{\alpha}}u \in L^1_{loc}(\R)$ exists, then 
             \begin{align} \label{WeakMollifier2}
                {^{\pm}}{\mathcal{D}}{^{\alpha}} u^\eps = \eta_{\eps} * {^{\pm}}{\mathcal{D}}{^{\alpha}} u \qquad\mbox{a.e. in } \R. 
            \end{align}
 
        \end{lemma}
        
        
        The next theorem gives a characterization of weak fractional derivatives on $\R$. 
        
         \begin{theorem}\label{characterization}
            Suppose $u \in L^{1}(\R)$. Then ${^{\pm}}{\mathcal{D}}{^{\alpha}} u = v \in L^{1}_{loc}(\R)$ exists if and only if there 
            exists a sequence 
              $\left\{u_j \right\}_{j=1}^{\infty} \subset C^{\infty}_{0}(\R)$ such that $u_j \rightarrow u$ in $L^{1}(\R)$ and ${^{\pm}}{\mathcal{D}}{^{\alpha}} u_j \rightarrow v$ in $L^{1}_{loc}(\R)$.
        \end{theorem}

        \begin{proof}
            {\em Step 1:}  Suppose that there exists $v \in L^{1}_{loc}(\R)$ and $\left\{u_j \right\}_{j=1}^{\infty} \subset C^{\infty}_{0}(\R)$
             such that $u _ j \rightarrow u$ in $L^{1}(\R)$ and ${^{\pm}}{\mathcal{D}}{^{\alpha}}u_j     \rightarrow v$ in $L^{1}_{loc}(\R)$. We want to show $v = {^{\pm}}{\mathcal{D}}{^{\alpha}} u$ almost everywhere. For any $\varphi\in C^\infty_0(\R)$  
            \begin{align*}
                \left| \int_{\R} (u-u_j )(x) {^{\mp}}{D}{^{\alpha}}\varphi(x) \,dx \right| &\leq \int_{\R} \left| (u - u_j)(x) \right| \left| {^{\mp}}{D}{^{\alpha}}\varphi (x) \right|\,dx\\
                &\leq \|u - u_j \|_{L^{1}(\R)} \left\|{^{\mp}}{D}{^{\alpha}} \varphi \right\|_{L^{\infty}(\R)} \to 0   
            \end{align*} 
       for $j\to \infty$ and for $K : = \supp(\varphi)$
            \begin{align*}
                \left| \int_{\R} \left(v - {^{\pm}}{\mathcal{D}}{^{\alpha}} u_j\right)(x) \varphi(x)\,dx  \right| &\leq \int_{\R} \left|\left(v - {^{\pm}}{\mathcal{D}}{^{\alpha}} u_j\right)(x)\right| |\varphi(x)|\,dx\\
                &= \int_{K} \left| \left(v - {^{\pm}}{\mathcal{D}}{^{\alpha}} u_j \right) (x) \right| \left|\varphi(x) \right| \,dx\\
                &\leq \left\| v - {^{\pm}}{\mathcal{D}}{^{\alpha}}u_{j}\right\|_{L^{1}(K)} \| \varphi \|_{L^{\infty}(\R)} \to 0   
            \end{align*}
         for $j\to \infty$. 
         From these inequalities and the definition of weak derivatives we get
            \begin{align*}
                (-1)^{[\alpha]} \int_{\R} u(x) {^{\mp}}{D}{^{\alpha}} \varphi (x)\,dx &= \lim_{j \rightarrow \infty} (-1)^{[\alpha]} \int_{\R} u_j(x) {^{\mp}}{D}{^{\alpha}}\varphi(x)\,dx \\ 
                &=\lim_{j \rightarrow \infty} \int_{\R} {^{\pm}}{\mathcal{D}}{^{\alpha}} u_j (x)  \varphi(x) \,dx 
                = \int_{\R} v(x) \varphi(x)\,dx. 
            \end{align*}
            By the uniqueness of the weak derivative, we deduce $v = {^{\pm}}{\mathcal{D}}{^{\alpha}} u$ almost everywhere.
            
            \smallskip
            {\em Step 2:} Suppose that $u \in L^{1}(\R)$ and $v := {^{\pm}}{\mathcal{D}}{^{\alpha}} u 
            \in L^{1}_{loc}(\R)$. We want to show that there exists $\left\{ u_j 
             \right\}_{j=1}^{\infty} \subset C^{\infty}_{0}(\R)$ such that $u_j \rightarrow u$ in $L^{1}(\R)$
             and ${^{\pm}}{\mathcal{D}}{^{\alpha}} u_j \rightarrow v$ in $L^{1}_{loc}(\R)$. 
            To the end, let $\psi \in C^{\infty}(\R)$ satisfy $\psi(t) =1$ if $t \leq 0$ and $\psi (t) = 0$ if 
            $t \geq 1$. For $j =1 ,2,3,...$ let $\psi_{j} \in C^{\infty}_{0}(\R)$ be defined by 
            $\psi_{j}(x) : = \psi(|x| - j)$. Let $u_j : =\eta_{\frac{1}{j}} * (\psi_{j}u)$. Then $u_j\in C^\infty_0(\R)$ and $u_j \rightarrow u$ in $L^{1}(\R)$ as $j \rightarrow \infty$.
            We also claim that ${^{\pm}}{\mathcal{D}}{^{\alpha}} u_j \rightarrow v$ in $L^{1}_{loc}(\R)$  as
            $j\to \infty$ and prove this conclusion below in the subsequent corollary.
       \end{proof}
   
   \medskip
   \begin{corollary}\label{corollary4.8}
   	Suppose $u \in L^{p}(\R)$ for $1 \leq p < \infty$. Then $v := {^{\pm}}{\mathcal{D}}{^{\alpha}} u \in L^{q}_{loc}(\R)$
   	for $1\leq q <\infty$ 
   	if and only if there exists $\left\{u_j \right\}_{j=1}^{\infty} \subset C^{\infty}_{0}(\R)$ such that $u_j \rightarrow u$ in $L^{p}(\R)$ and ${^{\pm}}{\mathcal{D}}{^{\alpha}} u_j \rightarrow v$ in $L^{q}_{loc}(\R)$.
   \end{corollary}

          \begin{proof}
            {\em Step 1:}  Same as {\em Step 1} of the proof of Theorem \ref{characterization}. 
            
            {\em Step 2:}  Suppose that $u \in L^{p}(\R)$ for $1 \leq p < \infty$ and $v := {^{\pm}}{\mathcal{D}}{^{\alpha}} u 
            \in L^{q}_{loc}(\R)$. Let $\{u_j\}_{j=1}^{\infty}$ be the same as in the proof of Theorem \ref{characterization} and $\eps>0$. We now want to show that ${^{\pm}}{\mathcal{D}}{^{\alpha}} u_j 
            \rightarrow v$ in $L^{q}_{loc}(\R)$ as $j\to \infty$.

            By the assumption, we have $v={^{\pm}}{\mathcal{D}}{^{\alpha}} u \in L^{q}_{loc}(\R)$. 
            For any fixed compact subset $K \subset \R$, choose $a,b \in \R$ such that $K \subset (a , b)$ finite. 
            By the construction of $u_j$, we have ${^{\pm}}{\mathcal{D}}{^{\alpha}} u_j
            =\eta_{\frac{1}{j}} * {^{\pm}}{\mathcal{D}}{^{\alpha}}(\psi_{j} u)$. 
            Let $K_j : = \text{supp}(\psi_j)$ and for every $\varphi\in C^{\infty}_{0}(\R)$ with 
            $\text{supp}(\varphi) \subset K_j$ we have  
            \begin{align*}
                \int_{K_j} {^{\pm}}{\mathcal{D}}{^{\alpha}}(\psi_j u)(x) \varphi(x)\,dx &= \int_{\R} {^{\pm}}{\mathcal{D}}{^{\alpha}}(\psi_j u)(x) \varphi(x)\,dx \\ 
                &=\int_{\R} (\psi_j u)(x) {^{\mp}}{D}{^{\alpha}} \varphi(x)\,dx 
                = \int_{K_j} (\psi_ju)(x) {^{\mp}}{D}{^{\alpha}} \varphi(x)\,dx. 
            \end{align*}
            Hence, ${^{\pm}}{\mathcal{D}}{^{\alpha}} (\psi_j u)$ can be regarded as the weak 
            fractional derivative of $\psi_j u$ over the domain $K_j$. It is due to this fact that we could 
            use the product rule with remainder for fractional weak derivatives (to be proved  in Theorem \ref{WeakProductRule})
            to get  
            \begin{align*}
                {^{\pm}}{\mathcal{D}}{^{\alpha}} (\psi_j u ) (x) = \psi_j (x) {^{\pm}}{\mathcal{D}}{^{\alpha}} u(x) + \sum_{k=1}^{m} C_{k} {^{\pm}}{I}{^{k-\alpha}} u(x) D^{k}\psi_{j}(x) + {^{\pm}}{R}{^{\alpha}_{m}}(u,\psi_j)(x). 
            \end{align*}
            Therefore, 
            \begin{align*}
                & \Bigl\| {^{\pm}}{\mathcal{D}}{^{\alpha}} u - {^{\pm}}{\mathcal{D}}{^{\alpha}} u_j\Bigr\|_{L^{q}(K)} 
                =\Bigl\| {^{\pm}}{\mathcal{D}}{^{\alpha}} u - \eta_{\frac{1}{j}} * {^{\pm}}{\mathcal{D}}{^{\alpha}} (\psi_{j} ,u)\Bigr\|_{L^{q}(K)}\\ 
                &\qquad =\Bigl\|{^{\pm}}{\mathcal{D}}{^{\alpha}} u - \eta_{\frac{1}{j}} * \Bigl( \psi_j {^{\pm}}{\mathcal{D}}{^{\alpha}} u + \sum_{k=1}^{m} C_{k} {^{\pm}}{I}{^{\alpha}} u D^{k} \psi_{j} + {^{\pm}}{R}{^{\alpha}_{m}} (u,\psi_j)\Bigr) \Bigr\|_{L^{q}(K)}\\
                &\qquad \leq \Bigl\|{^{\pm}}{\mathcal{D}}{^{\alpha}} u - \eta_{\frac{1}{j}} * \psi_{j} {^{\pm}}{\mathcal{D}}{^{\alpha}} u \Bigr\|_{L^{q}(K)} + \Bigl\| \eta_{\frac{1}{j}} * \sum_{k=1}^{m} C_{k} {^{\pm}}{I}{^{\alpha}} u D^{k} \psi_{j} \Bigr\|_{L^{q}(K)} \\
                &\hskip 1.0in + \Bigl\| {^{\pm}}{R}{^{\alpha}_{m}} (u,\psi_j) \Bigr\|_{L^{q}(K)}\\
                &\qquad  \leq \Bigl\|{^{\pm}}{\mathcal{D}}{^{\alpha}} u - \eta_{\frac{1}{j}} * \psi_{j} {^{\pm}}{\mathcal{D}}{^{\alpha}} u \Bigr\|_{L^{q}(K)} + \sum_{k=1}^{m} \Bigl\| \eta_{\frac{1}{j}} * C_{k} {^{\pm}}{I}{^{\alpha}} u D^{k} \psi_{j} \Bigr\|_{L^{q}(K)} \\
                &\hskip 1.0in + \left\| {^{\pm}}{R}{^{\alpha}_{m}} (u,\psi_j) \right\|_{L^{q}(K)}.
            \end{align*}
            Then it suffices to show that each of the above three terms vanishes as $j\to \infty$. 
            
            Since ${^{\pm}}{\mathcal{D}}{^{\alpha}} u \in L^{q}(K)$, by the same arguments used 
            to show that $u_j \rightarrow u$ in $L^{p}(\R)$, we have that $\eta_{\frac{1}{j}} * \psi_j {^{\pm}}{\mathcal{D}}{^{\alpha}} u \rightarrow {^{\pm}}{\mathcal{D}}{^{\alpha}} u$ in $L^{q}(K)$. Hence, there exists $J_{1} \in \N$ such that for every $j\geq J_{1}$, we have that 
            \begin{align*}
                \left\| {^{\pm}}{\mathcal{D}}{^{\alpha}} u - \eta_{\frac{1}{j}} * \psi_{j} {^{\pm}}{\mathcal{D}}{^{\alpha}} u \right\|_{L^{q}(K)} < \dfrac{\eps}{2}. 
            \end{align*}

            Next, by construction, for set $K$, there exists $J_{2}:= J_{2}(K) \in \N$ so that for every $j\geq J_{2}$, $D^{k} \psi_j(x) = 0$ for every $x \in (a,b)$. Therefore, for every $j \geq J_{2},$
            \begin{align*}
                \left\|\eta_{\frac{1}{j}} * C_{k} {^{\pm}}{I}{^{k-\alpha}} u D^{k}\psi_j \right\|_{L^{q}(K)} 
                \leq \left\|C_{k} {^{\pm}}{I}{^{k-\alpha}} u D^{k} \psi_j\right\|_{L^{q}((a,b))} = 0.
            \end{align*}
            Here we have used  the fact that for each $k$ and $j \geq J_{2}$, ${^{\pm}}{I}{^{k-\alpha}} u$ is finite on $(a,b)$. This of course is true since $u \in L^{p}(\R)$. In fact, we need only that ${^{\pm}}{I}{^{k-\alpha}} u$ is finite on $(a,b)$ for $j = J_{2}$ since for all $j \geq J_{2}$, $D^{k} \psi_{j} \equiv 0$ in $(a,b)$. 

            Finally, it can be shown that the remainder term vanishes as 
            $j\rightarrow \infty$. However, since the argument is rather lengthy, 
            	we omit it to save space and refer the reader  
            	to \cite[Corollary 4.2]{Feng_Sutton} for a complete argument.  
        \end{proof}

    \begin{remark}
    	The conclusion of the above corollary still holds if $L^q_{loc}(\R)$ is replaced by $L^q(\R)$ in the statement. 
    \end{remark}

    \subsection{Basic Properties of Weak Fractional Derivatives}\label{sec-4.4}  
    Similar to the classical calculus theory, we expect weak derivatives 
    to satisfy certain properties and rules 
    of calculus. As in the classical fractional calculus, many of the rules in the weak fractional calculus theory differ 
    from their integer counterparts, which is expected. Below we list a few elementary properties for weak fractional derivatives.
    
    \begin{proposition} \label{properties}
    	Let  $\alpha , \beta >0$, $\lambda ,\mu \in \R$, and $u,v$ be weakly differentiable to the appropriate order. Then the following properties hold.  
    \begin{itemize}
    \item[{\rm(i)}] Linearity: ${^{\pm}}{\mathcal{D}}{^{\alpha}}(\lambda u + \mu v) = \lambda {^{\pm}}{\mathcal{D}}{^{\alpha}} u + \mu {^{\pm}}{\mathcal{D}}{^{\alpha}} v$.
    \item[{\rm (ii)}] Inclusivity: Let $0 < \alpha < \beta < 1$, suppose that $u$ is $\beta$ order weakly differentiable. Then $u$ is $\alpha$ order weakly differentiable.    
     \item[{\rm (iii)}] Semigroup:  suppose $ 0<\alpha, \beta, \alpha+\beta<1$ and 
    ${^{\pm}}{\mathcal{D}}{^{\alpha}} u, {^{\pm}}{\mathcal{D}}{^{\beta}} u, {^{\pm}}{\mathcal{D}}{^{\alpha+\beta}} u\in L^1(\Omega)$, then ${^{\pm}}{\mathcal{D}}{^{\alpha}} {^{\pm}}{\mathcal{D}}{^{\beta}} u = {^{\pm}}{\mathcal{D}}{^{\alpha+\beta}} u$. Moreover, if $\alpha>1$, then  ${^{\pm}}{\mathcal{D}}{^{\alpha}} u
    ={^{\pm}}{\mathcal{D}}{^{[\alpha]+\sigma}} u=  {\mathcal{D}}{^{[\alpha]}} ({^{\pm}}{\mathcal{D}}{^{\sigma}} u)$ with $\sigma:=\alpha- [\alpha]$.
    \item[{\rm (iv)}] Consistency: if $u$ is first-order weakly differentiable, then the $\alpha\, (<1)$ order weak derivative coincides with the first-order weak derivative in the limit as $\alpha\to 1$.
  
    \end{itemize}
    \end{proposition}
    
    \begin{proof}
    (i) 
	It follows straightforwardly from a direction computation. 
	
	(ii) We shall postpone this proof until after the Fundamental Theorem of Weak Fractional Calculus (cf. Theorem \ref{WeakFTFC}) is established.
	
	(iii) If $0<\alpha,\beta, \alpha+\beta<1$, by the definition we have 
		\begin{align}\label{ee1}
		\int_{\Omega} {^{\pm}}{\mathcal{D}}{^{\alpha+\beta}} u\, \varphi\,dx
		&=  \int_{\Omega} u {^{\mp}}{D}{^{\alpha + \beta}} \varphi\,dx \qquad\forall \varphi\in C^\infty_0(\Omega), \\
		\int_{\Omega} {^{\pm}}{\mathcal{D}}{^{\beta}} {^{\pm}}{\mathcal{D}}{^{\alpha}} u\, \varphi \,dx 
		&=  \int_{\Omega} {^{\pm}}{\mathcal{D}}{^{\alpha}} u {^{\mp}}{D}{^{\beta}} \varphi\,dx  \qquad\forall \varphi\in C^\infty_0(\Omega).\label{ee2} 
		\end{align}
		Let $\{u_j\}_{j=1}^{\infty} \subset C^\infty (\Omega)$ such that $u_j\to u$ in $L^1(\Omega)$ and ${^{\pm}}{D}{^{\alpha}}u_j \to {^{\pm}}{D}{^{\alpha}}u$ in $L^1(\Omega)$, then using the integration by parts formula for Riemann-Liouville fractional order derivatives, we obtain 
		\begin{align}\label{ee3}
		\int_{\Omega} {^{\pm}}{\mathcal{D}}{^{\alpha}} u {^{\mp}}{D}{^{\beta}} \varphi\,dx
		&= \lim_{j \rightarrow \infty} \int_{\Omega} {^{\pm}}{D}{^{\alpha}}u_j {^{\mp}}{D}{^{\beta}} \varphi\,dx\\
		&= \lim_{j \rightarrow \infty}  \int_{\Omega}  u_j {^{\mp}}{D}{^{\alpha}}{^{\mp}}{D}{^{\beta}} \varphi\,dx \nonumber\\
		&= \lim_{j \rightarrow \infty}  \int_{\Omega}  u_j {^{\mp}}{D}{^{\alpha+\beta}}  \varphi\,dx
		=   \int_{\Omega}  u {^{\mp}}{D}{^{\alpha+\beta}}  \varphi\,dx. \nonumber
 		\end{align}
 		Combining \eqref{ee1}--\eqref{ee3} we get 
 		\[
 		\int_{\Omega} {^{\pm}}{\mathcal{D}}{^{\alpha+\beta}} u \varphi\,dx
 		= \int_{\Omega} {^{\pm}}{\mathcal{D}}{^{\beta}} {^{\pm}}{\mathcal{D}}{^{\alpha}} u \varphi\,dx
 		\qquad\forall \varphi\in C^\infty_0(\Omega).
 		\]
 		Thus, ${^{\pm}}{\mathcal{D}}{^{\alpha+\beta}} u=  {^{\pm}}{\mathcal{D}}{^{\beta}} {^{\pm}}{\mathcal{D}}{^{\alpha}} u$ almost everywhere in $\Omega$. 
 		
 		If $\alpha>1$, set $m=[\alpha]$ and $\sigma=\alpha-m$. By the definition we get for any $\varphi\in C^\infty_0(\Omega)$,
 		\begin{align*}
 		\int_{\Omega} {\mathcal{D}}{^{[\alpha]}} ({^{\pm}}{\mathcal{D}}{^{\sigma}} u) \varphi\,dx
 		&=(-1)^{[\alpha]} \int_{\Omega}   {^{\pm}}{\mathcal{D}}{^{\sigma}} u \, {\mathcal{D}}{^{[\alpha]}} \varphi\,dx
 		=(-1)^{[\alpha]} \int_{\Omega}     u  \, {^{\mp}}{\mathcal{D}}{^{\sigma}}  {\mathcal{D}}{^{[\alpha]}} \varphi\,dx \\
 		&=(-1)^{[\alpha]} \int_{\Omega}     u  \, {^{\mp}}{\mathcal{D}}{^{\sigma +[\alpha]}} \varphi\,dx 
 		=(-1)^{[\alpha]} \int_{\Omega}  u  \, {^{\mp}}{\mathcal{D}}{^\alpha} \varphi\,dx.
 		\end{align*}
 		Thus, ${^{\pm}}{\mathcal{D}}{^{\alpha}} u
 		 =  {\mathcal{D}}{^{[\alpha]}} ({^{\pm}}{\mathcal{D}}{^{\sigma}} u)$ almost everywhere in $\Omega$ and
        the assertion (iii) is proved. 
 
	  (iv)  It follows by the consistency of the classical fractional derivatives that for every $\varphi \in C^{\infty}(\Omega)$, 
		\begin{align*}
			\lim_{\alpha \rightarrow 1} \int_{\Omega} {^{\pm}}{\mathcal{D}}{^{\alpha}} u\, \varphi\,dx &:= \lim_{\alpha \rightarrow 1} (-1)^{[\alpha]} \int_{\Omega} u {^{\mp}}{D}{^{\alpha}} \tilde{\varphi} \,dx
			 = - \int_{\Omega} u D\tilde{\varphi} \,dx
			=:  \int_{\Omega} \mathcal{D} u\, \varphi \,dx
		\end{align*}

    \end{proof}
    
    \begin{remark}
    	 We note that for $\alpha>1$,  generally, ${^{\pm}}{\mathcal{D}}{^{\alpha}} u
    	 \neq  {^{\pm}}{\mathcal{D}}{^{\sigma}}  {\mathcal{D}}{^{[\alpha]}}  u$, consequently, 
    	 ${^{\pm}}{\mathcal{D}}{^{\sigma}}  {\mathcal{D}}{^{[\alpha]}}  u  \neq {\mathcal{D}}{^{[\alpha]}} {^{\pm}}{\mathcal{D}}{^{\sigma}}   u$, 
    	  in general.
    \end{remark}
    
    We conclude this section by stating a general integration by parts formula in the case $\Ome=\R$. 
    
    \begin{proposition}
    	Let $\alpha>0$, $1\leq p_k\leq \infty$ and $q_k=\frac{p_k}{p_k-1}$ for $k=1,2$. Suppose that $u \in L^{p_1}(\R)$, $v\in L^{p_2}(\R)$,  ${^{\pm}}{\mathcal{D}}{^{\alpha}} u \in L^{q_2}(\R)$, and 
    		$ {^{\mp}}{\mathcal{D}}{^{\alpha}} v \in L^{q_1}(\R)$. Then there holds 
    	 \begin{align}\label{integration_by_parts} 
    	\int_{\R} {^{\pm}}{ \mathcal{\mathcal{D} }}{^{\alpha}} u\, v\,dx  
    	= (-1)^{[\alpha]} \int_{\R} u\, {^{\mp}}{\mathcal{D}}{^{\alpha}} v\,dx . 
    	\end{align}
    \end{proposition}
    
    \begin{proof}
    	By Corollary \ref{corollary4.8} we know that there exists a sequence $\{v_j\}_{j =1}^{\infty} \subset C^\infty_0(\R)$ 
    	such that $v_j\to v$ in $L^{p_2}(\R)$ and ${^{\mp}}{ \mathcal{\mathcal{D} }}{^{\alpha}} v_j \to 
    	{^{\mp}}{ \mathcal{\mathcal{D} }}{^{\alpha}} v $ in $L^{q_1}(\R)$ as $j\to \infty$. 
    	By the definition of $ {^{\pm}}{ \mathcal{\mathcal{D} }}{^{\alpha}} u$ we have 
    	 \begin{align*}
    	\int_{\R}  {^{\pm}}{ \mathcal{\mathcal{D} }}{^{\alpha}} u\, v_j\,dx 
    	= (-1)^{[\alpha]} \int_{\R} u\, {^{\mp}}{\mathcal{D} }{^{\alpha}} v_j\,dx . 
    	\end{align*}
    	Setting $j\to \infty$ immediately infers \eqref{integration_by_parts}. The proof is complete. 
    \end{proof} 

We note that in order to extend the above integration by parts  formula to the finite domain case, 
	it requires the notion of function traces, both function traces and the extended 
	formula will 
	be presented in \cite{Feng_Sutton1a} for functions in fractional Sobolev spaces.

    \subsection{Product and Chain Rules for Weak Fractional Derivatives}\label{sec-4.5}
    In this subsection we present some product and chain rules for weak fractional derivatives, which are
    similar to those for classical fractional derivatives given in \cite[Section 2.5]{Feng_Sutton}.

\begin{theorem}\label{WeakProductRule}
    Let $(a,b) \subset \R$ and $0<\alpha <1$.  Suppose that $\psi \in C^{m+1}([a,b])$ for $m \geq 1$ and ${^{\pm}}{\mathcal{D}}{^{\alpha}} u\in L^1_{loc}((a,b))$ exists. Then  ${^{\pm}}{\mathcal{D}}{^{\alpha}} (u \psi)$ 
    exists and is given by 
    \begin{align*}
        {^{\pm}}{\mathcal{D}}{^{\alpha}} (u \psi)(x) = {^{\pm}}{\mathcal{D}}{^{\alpha}} u(x) \cdot \psi(x) &+ \sum_{k=1}^{m} \dfrac{\Gamma(1 + \alpha)}{\Gamma(1+ k)\Gamma(1 - k + \alpha)} {^{\pm}}{I}{^{k-\alpha}} u(x) D^{k} \psi(x) \\
        &+ {^{\pm}}{R}{^{\alpha}_{m}}(u, \psi)(x) \qquad \mbox{a.e. in } (a,b),
    \end{align*}
    where 
    \begin{align*}
        {^{+}}{R}{^{\alpha}_{m}} (u,\psi)(x) = \dfrac{(-1)^{m+1}}{m! \Gamma(-\alpha)} \int_{x}^{b} \dfrac{u(y)}{(y -x )^{1 + \alpha}} \, dy \int_{x}^{y} \psi^{(m+1)} (z) (z - x)^{m} \, dz,\\
        {^{-}}{R}{^{\alpha}_{m}}(u,\psi)(x) = \dfrac{(-1)^{m+1}}{m! \Gamma(-\alpha)} \int_{a}^{x} \dfrac{u(y)}{(x-y)^{1 + \alpha}} \, dy \int_{y}^{x} \psi^{(m+1)}(z) (x -z)^{m} \, dz.
    \end{align*}
\end{theorem}

\begin{proof}
    Let $\left\{u_j \right\}_{j=1}^{\infty} \subset C^{\infty}((a,b))$ so that $u_j \rightarrow u$ in $ L^{1}((a,b))$ 
   and ${^{\pm}}{\mathcal{D}}{^{\alpha}} u_j \rightarrow {^{\pm}}{\mathcal{D}}{^{\alpha}} u$ in $
     L^{1}_{loc}((a,b))$. Consider the product $u_j \psi$, which belongs to $C((a,b))$,  and $\varphi \in C^{\infty}_{0}((a,b))$ with $\supp(\varphi) : = (c,d) \subset (a,b)$.  
      Since $u_j \rightarrow u$ in  $L^{1}(a,b))$, ${^{\pm}}{I}{^{\sigma}} u_j \rightarrow 
      {^{\pm}}{I}{^{\sigma}} u$ in $L^{1}((a,b))$. Using this fact and \cite[Theorem 2.3]{Feng_Sutton}, we obtain  
    \begin{align*}
        \int_{\Omega} u \psi {^{\mp}}{D}{^{\alpha}}\varphi\,dx
        &= \lim_{j \rightarrow \infty} \int_{\Omega} u_j \psi {^{\mp}}{D}^{{\alpha}} \varphi \,dx \\
        &= \lim_{j\rightarrow \infty} \int_{\Omega} {^{\pm}}{D}{^{\alpha}} (u_j \psi) \cdot \varphi \,dx
        = \lim_{j\rightarrow \infty} \int_{\Omega'} {^{\pm}}{D}{^{\alpha}} (u_j \psi) \cdot \varphi\,dx \\
        &= \lim_{j \rightarrow \infty} \int_{\Omega'} \left({^{\pm}}{D}{^{\alpha}} u_j \cdot \psi + \sum_{k=1}^{m} C_{k} {^{\pm}}{I}{^{k-\alpha}} u_j D^{k} \psi + {^{\pm}}{R}{^{\alpha}_{m}}(u_j,\psi) \right) \varphi\,dx \\ 
        &= \int_{\Omega} \left({^{\pm}}{\mathcal{D}}{^{\alpha}} u \cdot \psi + \sum_{k=1}^{m} C_{k} {^{\pm}}{I}{^{k-\alpha}} u D^{k} \psi + {^{\pm}}{R}{^{\alpha}_{m}(u,\psi)} \right) \varphi\,dx,
    \end{align*}
    which implies that 
    \begin{align*}
        {^{\pm}}{\mathcal{D}}{^{\alpha}} (u\psi) = {^{\pm}}{\mathcal{D}}{^{\alpha}} u \cdot \psi + \sum_{k=1}^{m} \dfrac{\Gamma(1+ \alpha)}{\Gamma(k+1) \Gamma(1 - k + \alpha)} {^{\pm}}{I}{^{k-\alpha}} u D^{k} \psi + {^{\pm}}{R}{_{m}^{\alpha}}(u,\psi)
    \end{align*}
    almost everywhere in $(a,b)$ with 
    \begin{align*}
        {^{+}}{R}{^{\alpha}_{m}} (u,\psi) (x) &= \dfrac{(-1)^{m}}{m! \Gamma(-\alpha)} \int_{x}^{b} \dfrac{u(y)}{(y -x )^{1 + \alpha}} \, dy \int_{x}^{y} \psi^{(m+1)} (z) (z - x)^{m} \, dz, \\
        {^{-}}{R}{^{\alpha}_{m}}(u,\psi)(x) & = \dfrac{(-1)^{m+1}}{m! \Gamma(-\alpha)} \int_{a}^{x} \dfrac{u(y)}{(x-y)^{1 + \alpha}} \, dy \int_{y}^{x} \psi^{(m+1)}(z) ( x-z)^{m} \, dz.
    \end{align*}
    The proof is complete.
\end{proof}

\begin{remark}
We also can prove another version of the product rules that do not include remainder terms. 
That version of the 
product rules will instead be written as infinite sums and require 
both functions are analytic. Because we do not wish to make such 
an assumption in our applications of the  weak fractional derivative product rule, 
we omit that version of the product rules. 
\end{remark}
 
Based on the above product rules with $m=0$, we can easily obtain this following chain rules for weak fractional 
derivatives. We omit the proof because it is similar to the proof of \cite[Theorem 2.4]{Feng_Sutton}. 

\begin{theorem}\label{thm_CR_weak}
	Let $(a,b)\subset \R$. Suppose that $\varphi \in C^1(\R)$ such that $\varphi(0)=0$ and $f\in C((a,b))$. Then there hold
	\begin{align}\label{chain_rule_weak}
	{^{\pm}}{\mathcal{D}}{^{\alpha}}  \varphi(f)(x)  =   \frac{\varphi(f)(x)}{f(x)} {^{\pm}}{\mathcal{D}}{^{\alpha}} f(x)
	+ {^{\pm}}{R}{_{0}^{\alpha}}\Bigl(f, \frac{\varphi(f)}{f} \Bigr)(x) \qquad \mbox{a.e. in } (a,b),
	\end{align}
	where ${^{\pm}}{R}{_{0}^{\alpha}}(f, g)$  are defined by 
	\begin{align}\label{LeftProductRuleRemainder_0}
    	        {^{-}}{R}{^{\alpha}_{0}}(f,g)(x) = \dfrac{-1}{\Gamma(- \alpha)} \int_{a}^{x} \dfrac{f(y) [g(x)-g(y)]}{(x- y)^{1 + \alpha}} \, dy , \\
    	        \label{RightProductRuleRemainder_0}
    	        {^{+}}{R}{^{\alpha}_{0}} (f,g)(x) = \dfrac{-1}{\Gamma(-\alpha)} \int_{x}^{b} \dfrac{f(y) [g(x)-g(y)]}{(y-x)^{1+\alpha}}\,dy .
    	    \end{align}
	
\end{theorem}

\subsection{Fundamental Theorem of Weak Fractional Calculus (FTwFC)}\label{sec-4.6}

In this subsection, we aim to extend the FTcFC (see Theorem \ref{FTFC}) to weakly fractionally differentiable functions. Similar to the FTcFC for classical fractional (Riemann-Liouville) derivatives, the finite and infinite domain cases are significantly different, hence must be treated separately. 

\subsubsection{\bf The Finite Interval Case}\label{sec-4.6.1}
To establish the FTwFC on a finite domain, 
we first need to show the following crucial lift lemma.
   
   \begin{lemma}\label{FTFCConstant}
    Let $\Omega\subset \R$ and $0<\alpha <1$. Suppose that  $u \in L^{p}(\Omega)$ and  ${^{\pm}}{\mathcal{D}}{^{\alpha}}u \in L^{p}(\Omega)$ for some $1 \leq p < \infty$. 
    Then ${^{\pm}}{I}{^{1-\alpha}}u \in  W^{1,1}(\Omega)$. 
    \end{lemma}
   
   \begin{proof}
       Choose $\{u_j\}_{j=1}^{\infty} \subset C^{\infty}(\Omega)$ so that $u_j \rightarrow u$ in $L^{p}(\Omega)$ and ${^{\pm}}{\mathcal{D}}{^{\alpha}}u_j \rightarrow {^{\pm}}{\mathcal{D}}{^{\alpha}}u$ in $L^{p}(\Omega)$. Since ${^{\pm}}{D}{^{\alpha}}u_j \in L^{1}(\Omega)$, 
       then ${^{\pm}}{I}{^{1-\alpha}} u_j \in W^{1,1}(\Omega) $.  
       By the stability property of ${^{\pm}}{I}{^{1-\alpha}}$ we have
       \begin{align*}
           \|{^{\pm}}{I}{^{1-\alpha}} u_m - {^{\pm}}{I}{^{1-\alpha}} u_n\|_{ W^{1,1}(\Omega)}         
           &=\|{^{\pm}}{I}{^{1-\alpha}} u_m - {^{\pm}}{I}{^{1-\alpha}} u_n\|_{L^{1}(\Omega)} \\
           &\qquad  + \|{^{\pm}}{D}{^{\alpha}} u_m - {^{\pm}}{D}{^{\alpha}} u_n \|_{L^{1}(\Omega) } \nonumber \\ 
           &\leq C \|u_m - u_n\|_{L^{1} (\Omega)} + \| {^{\pm}}{\mathcal{D}}{^{\alpha}}u_m - {^{\pm}}{\mathcal{D}}{^{\alpha}} u_n \|_{L^{1}(\Omega)}\\
           &\to 0\quad\mbox{as } n,m\to \infty.
       \end{align*}
       Hence,  $\{{^{\pm}}{I}{^{1-\alpha}}u_j\}_{j=1}^{\infty}$ is a Cauchy sequence in $W^{1,1}(\Omega)$. Since $W^{1,1}(\Omega)$ is a Banach space, there exists $v \in W^{1,1}(\Omega)$ so that ${^{\pm}}{I}{^{1-\alpha}}u_j \rightarrow v$ in $W^{1,1}(\Omega)$ . 
       
       It remains to show that $v = {^{\pm}}{I}{^{1-\alpha}}u$. On noting that  
       \begin{align*}
        \|v - {^{\pm}}{I}{^{1-\alpha}}u\|_{L^1(\Omega)}
        &  \leq \|v - {^{\pm}}{I}{^{1-\alpha}}u_j \|_{L^{1}(\Omega)} + \|{^{\pm}}{I}{^{1-\alpha}} u_j - {^{\pm}}{I}{^{1-\alpha}} u \|_{L^{1}(\Omega)}  \\
        &\leq \|v - {^{\pm}}{I}{^{1-\alpha}}u_j \|_{L^{1}(\Omega)} + C \| u_j -  u \|_{L^{1}(\Omega)} 
        \to 0 \quad \mbox{as } j\to \infty.
       \end{align*}
       Hence,  $v = {^{\pm}}{I}{^{1-\alpha}} u$ almost everywhere in $\Omega$. The proof is complete.
   \end{proof}

	\begin{remark}
	Since $AC(\overline{\Omega})$ is isomorphic to $W^{1,1}(\Omega)$ in the 1D case, 
	the above lemma also implies that ${^{\pm}}{I}{^{1-\alpha}}u \in AC(\overline{\Omega})$. The above lemma shows that 
	if $u \in {^{\pm}}{W}{^{\alpha,1}}(\Omega)$ (see the space definition in \cite{Feng_Sutton1a}),
	then the operator ${^{\pm}}{I}{^{1-\alpha}}$ lifts $u$ from 
	${^{\pm}}{W}{^{\alpha,1}}(\Ome)$ into ${W}{^{1,1}}(\Ome)$. This result reinforces the characterization of weakly fractional differentiable functions as stated in Section \ref{sec-4.2}. In particular, 
	one can roughly think about weakly fractional differentiable functions as those whose classical fractional derivatives exist almost everywhere. This is (almost) exactly the same characterization 
	for first order weakly differentiable functions (in 1D). Precisely, absolute continuity characterizes weakly differentiable functions and the absolute continuity of 
	${^{\pm}}{I}{^{1-\alpha}}u$ characterizes weakly fractional differentiable 
	functions. 
	\end{remark}

   \begin{theorem}\label{FTWFC}
       Let $\Omega\subset \R$ and $0 < \alpha <1$. Suppose that $u \in L^{p}(\Omega)$ and  ${^{\pm}}{\mathcal{D}}{^{\alpha}}u \in L^{p}(\Omega)$ for some $1\leq p < \infty$. Then
       there holds
       \begin{align}\label{WeakFTFC}
           u = c^{1-\alpha}_{\pm} \kappa^{\alpha}_{\pm}  + {^{\pm}}{I}{^{\alpha}}{^{\pm}}{\mathcal{D}}{^{\alpha}} u \qquad
           \mbox{a.e. in } \Omega.
       \end{align}
  
   \end{theorem}

   \begin{proof}
      Let $\{u_j \}_{j=1}^{\infty} \subset C^{\infty}(\Omega)$ so that $u_j \rightarrow u$ 
      in $L^{p}(\Omega)$ and ${^{\pm}}{\mathcal{D}}{^{\alpha}}u_j \rightarrow {^{\pm}}{\mathcal{D}}{^{\alpha}}u$ in $L^{p}(\Omega)$; in particular, 
      $u_j$ and its derivative converge in $L^{1}(\Omega)$. By Lemma \ref{FTFCConstant}, ${^{\pm}}{I}{^{1-\alpha}}u_j \rightarrow {^{\pm}}{I}{^{1-\alpha}} u$ in $W^{1,1}(\Ome)\cong AC(\overline{\Omega})$.
      Moreover, by the FTcFC we get
      \[
      u_j (x) = c^{1-\alpha}_{j,\pm} \kappa^{\alpha}_{\pm}(x)  + {^{\pm}}{I}{^{\alpha}}{^{\pm}}{D}{^{\alpha}} u_j (x).
      \]
Thus,
      \begin{align*}
     & \|u-c^{1-\alpha}_{\pm} \kappa^{\alpha}_{\pm} - {^{\pm}}{I}{^{\alpha}}  {^{\pm}}{D}{^{\alpha}} u\|_{L^1(\Omega)} \\
         &\quad\leq  \|u - u_j\|_{L^{1}(\Omega)} + |c^{1-\alpha}_{\pm} - c^{1-\alpha}_{j,\pm}| \,\|\kappa^{\alpha}_{\pm}\|_{L^{1}(\Omega)} + \|{^{\pm}}{I}{^{\alpha}} ({^{\pm}}{\mathcal{D}}{^{\alpha}} u - {^{\pm}}{D}{^{\alpha}} u_j ) \|_{L^{1}(\Omega)}\\
         &\quad\leq  \|u - u_j\|_{L^{1}(\Omega)} + |c^{1-\alpha}_{\pm} - c^{1-\alpha}_{j,\pm}| \,\|\kappa^{\alpha}_{\pm}\|_{L^{1}(\Omega)} + C\|{^{\pm}}{\mathcal{D}}{^{\alpha}} u - {^{\pm}}{D}{^{\alpha}}u_j\|_{L^{1}(\Omega)} \\
         &\quad \to 0 \qquad\mbox{as } j\to \infty.
      \end{align*}
     Therefore, 
     \[
     u -c^{1-\alpha}_{\pm} \kappa^{\alpha}_{\pm} - {^{\pm}}{I}{^{\alpha}}  {^{\pm}}{D}{^{\alpha}} u =0 \qquad \mbox{a.e. in }  \Omega.
     \]
  The proof is complete. 
   \end{proof}

\begin{remark}
	(a) We refer to Theorem \ref{FTWFC} as the Fundamental Theorem of Weak Fractional 
	Calculus (FTwFC) in this paper.  
	
	(b) The above FTwFC is an essential tool for studying weakly fractional differentiable functions, 
	in particular, it will play a crucial role in proving compact and Sobolev embeddings and a fractional Poincar\'e inequality in  \cite{Feng_Sutton1a}.
\end{remark}
   
    To conclude this subsection,  we would like to circle back to an unproven 
    inclusion result for weak fractional derivatives which was alluded to in 
    Proposition \ref{properties} part $(ii)$. This then presents the first application of the FTwFC. 
   
   \begin{proposition}
       Let $\Omega \subset \R$ and $0 < \alpha < \beta <1$. Suppose that ${^{\pm}}{\mathcal{D}}{^{\beta}}u$ exists in $L^{1}(\Omega)$. Then ${^{\pm}}{\mathcal{D}}{^{\alpha}}u$ exists in $L^{1}(\Omega)$.
   \end{proposition}
   
   \begin{proof}
        It follows by Theorem \ref{FTWFC} that 
        \begin{align*}
            u = c^{1-\beta}_{\pm} \kappa^{\beta}_{\pm} + {^{\pm}}{I}{^{\beta}} {^{\pm}}{\mathcal{D}}{^{\beta}} u \qquad\mbox{a.e. in } \Omega.
        \end{align*}
         Then there holds   
        \begin{align*}
            \int_{\Omega} u\, {^{\mp}}{D}{^{\alpha}} \varphi\,dx 
            &= \int_{\Omega} \bigl(c^{1-\beta}_{\pm} \kappa^{\beta}_{\pm} + {^{\pm}}{I}{^{\beta}} {^{\pm}}{\mathcal{D}}{^{\beta}} u \bigr)\,{^{\mp}}{D}{^{\alpha}} \varphi\,dx \\
            &= \int_{\Omega} {^{\pm}}{D}{^{\alpha}} \bigl(c^{1-\beta}_{\pm} \kappa^{\beta}_{\pm} + {^{\pm}}{I}{^{\beta}} {^{\pm}}{\mathcal{D}}{^{\beta}} u \bigr)\,\varphi\,dx \\ 
            &= \int_{\Omega} \bigl(c^{1-\beta}_{\pm} \kappa^{\beta -\alpha}_{\pm} + {^{\pm}}{I}{^{\beta -\alpha}} {^{\pm}}{\mathcal{D}}{^{\beta}} u \bigr)\, \varphi\,dx. 
        \end{align*}
        Since a direct calculation shows that $v:=c^{1-\beta}_{\pm} \kappa^{\beta -\alpha}_{\pm} + {^{\pm}}{I}{^{\beta -\alpha}} {^{\pm}}{\mathcal{D}}{^{\beta}} u\in L^1(\Omega)$,
        then the above identity implies that ${^{\pm}}{\mathcal{D}}{^{\alpha}}u$ exists and 
        ${^{\pm}}{\mathcal{D}}{^{\alpha}}u=v$ almost everywhere in $\Omega$. The proof is complete.
   \end{proof}
   
\subsubsection{\bf The Infinite Interval Case}\label{sec-4.6.2}
Unlike the finite domain, the absence of any boundary in the infinite interval case 
$\Ome=\R$ allows for a cleaner statement of the FTwFC and a simpler proof.

\begin{theorem}\label{FTWFCR}
    Let $0 < \alpha < 1$. Suppose that $u, v \in L^{1}(\R)$. If
    \begin{align}\label{WeakFTFCR}
        u = {^{\pm}}{I}{^{\alpha}} v \quad \mbox{a.e. in } \R,
    \end{align}
    then ${^{\pm}}{\mathcal{D}}{^{\alpha}}u = v$ almost everywhere. The converse is true under the additional assumption $u(x) \rightarrow 0$ almost everywhere as $|x| \to \infty$.
\end{theorem}

\begin{proof}
    The assertion and the accompanying equation \eqref{WeakFTFCR} follow from
    an application of the characterization theorem (cf. Theorem \ref{characterization})
    for weak fractional derivatives and Theorem \ref{FTFCa}.
\end{proof}

As was illustrated in the finite domain case and the infinite domain case for classical fractional derivatives, we can use the relation \eqref{WeakFTFCR}
 to show a basic inclusion result for weak fractional derivatives.

\begin{proposition}
    Let $0 < \alpha < \beta < 1$. Suppose that $u, {^{\pm}}{\mathcal{D}}{^{\beta}} u \in L^{1}(\R)$. Then ${^{\pm}}{\mathcal{D}}{^{\alpha}}u$ exists in $L^{1}(\R)$. 
\end{proposition}

\begin{proof}
    Apply the characterization theorem for weakly fractional differentiable functions on $\R$ (cf. Theorem \ref{characterization}) and the FTwFC (cf. Theorem \ref{FTWFCR}), then pass limits.
\end{proof}



\section{Weak Fractional Derivatives of Distributions} \label{sec-6}
The aim of this section is to introduce some weak fractional derivative notions for distributions. 
Like in the integer order case, such a notion  is necessary in order to define fractional order 
weak derivatives for ``all functions" including very rough ones and will also provide a useful tool 
for studying  fractional order differential equations (cf. \cite{Feng_Sutton2, Guo, Meerschaert}).   

The main difficulty for doing so is caused by the 
pollution effect of fractional order derivatives (and integrals), as a result,  the standard test space $\mathscr{D}(\Omega):=C^\infty_0(\Omega)$ is 
not invariant under the mappings ${^{\pm}}{D}{^{\alpha}}$, instead, ${^{\pm}}{D}{^{\alpha}}
(\mathscr{D}(\Omega)) \subset {^{\pm}}{\mathscr{D}(\Omega)}:= {^{\mp}}{C}{^{\infty}_{0}}(\Omega)$ 
(see the definitions below). 
Hence, ${^{\pm}}{D}{^{\alpha}} \varphi$ become invalid test functions (or inputs) for a distribution 
$u\in \mathscr{D}'(\Omega)$  although $\varphi \in \mathscr{D}(\Omega)$ is.  To circumvent this difficulty, 
there are two approaches used in the literature. The first one, which is most popular \cite{Samko}, is 
to use different test spaces which are larger than the standard test space $\mathscr{D}(\Omega)$ so that the chosen test space is invariant under  the mappings ${^{\pm}}{D}{^{\alpha}}$, and then to consider generalized functions (still called distributions) as continuous linear functionals on the chosen 
test space. The second approach is to extend the domain of a distribution  
$u\in \mathscr{D}'(\Omega)$ without changing the standard test space $\mathscr{D}(\Omega)$
so that the extended distribution $\tilde{u}$ can take the inputs ${^{\pm}}{D}{^{\alpha}} \varphi$. 
In this section, we use both approaches although we give more effort to the second one because it 
covers general distributions in $\mathscr{D}'(\Omega)$,  not just a subclass of $\mathscr{D}'(\Omega)$.   

\subsection{Test Spaces, Distributions and One-sided Distributions}\label{sec-6.1}
  We first recall some of the necessary function spaces and notions of convergence that are inherent to 
  constructing a fractional derivative for distributions. We also introduce two new  spaces of 
  one-side compactly supported functions and establish some properties of the weak fractional derivative operators ${^{\pm}}{\mathcal{D}}{^{\alpha}}$ 
  on the new spaces. Unless stated otherwise, in this section $\Omega$ denotes  
  either a finite interval $(a,b)$ or the real line $\R$.  
  
    \begin{definition}
        Let $\mathscr{D}(\Omega): =C^{\infty}_{0}(\Omega)$ which is equipped with the following 
        topology (sequential convergence): given a sequence $\{\varphi_k\}_{k=1}^{\infty} \subset \mathscr{D}(\Omega)$ is said to converge to $\varphi \in \mathscr{D}(\Omega)$ if 
        \begin{itemize}
            \item[(a)] there exists a compact subset $K\subset \Omega$ such that $\supp(\varphi_k) \subset K$ for every $k$,
            \item[(b)] $D^{m} \varphi_k \rightarrow D^{m}\varphi$ uniformly in $K$ for each $m\geq 0$.
        \end{itemize}
    Let $\mathscr{D}^{\prime}(\Omega)$ denote the space of continuous linear functionals on 
    $\mathscr{D}(\Omega)$, namely the dual space. 
    Every functional in $\mathscr{D}^{\prime} (\Omega)$ is called a distribution. 
    \end{definition}

    \begin{definition}
        Define the following two spaces of one-side compactly supported functions: 
        \begin{align*}
        {^{-}}{\mathscr{D}(\Omega)} &:= \{\varphi\in C^{\infty}(\Omega)\,:\, \exists x_0 \in \Omega, \varphi(x) \equiv 0 \, \forall x \leq x_0\}, \\
         {^{+}}{\mathscr{D}(\Omega)} &:= \{\varphi\in C^{\infty}(\Omega)\,:\, \exists x_0 \in \Omega, \varphi(x) \equiv 0 \, \forall x \geq x_0\},
         \end{align*} 
        which are equipped with the following topology: given a sequence $\{\varphi_k\}_{k=1}^{\infty} \subset {^{\pm}}{\mathscr{D}(\Omega)}$, it is said to converge to 
        $\varphi \in  {^{\pm}}{\mathscr{D}(\Omega)}$ if 
        \begin{itemize}
            \item[(a)] there exists an $x_0 \in \Omega$ such that $\varphi_k (x) \equiv 0$ for all 
            $x \leq x_0$ (or $x \geq x_0$ in the case of the right space) for  $k\geq  1$, 
            \item[(b)] $D^m \varphi_k \rightarrow D^m \varphi$ uniformly in $\Omega$ for every $m\geq 0$.
        \end{itemize}
    Let ${^{\pm}}{\mathscr{D}^{\prime}(\Omega)} $ denote respectively the spaces of continuous linear functionals on 
    ${^{\pm}}{\mathscr{D}(\Omega)}$,  namely the dual spaces of ${^{\pm}}{\mathscr{D}}(\Omega)$. 
    Every functional in ${^{\pm}}{\mathscr{D}^{\prime}}(\Omega) $ is called a one-sided distribution. 
    \end{definition} 
    
    
   \begin{lemma}
      $\mathscr{D}(\Omega)$ and ${^{\pm}}{\mathscr{D}}(\Omega) $ are complete topological vector spaces
      and $\mathscr{D}(\Omega)\subset {^{\pm}}{\mathscr{D}(\Omega)} $.
   \end{lemma}

Recall that it was proved in Section \ref{sec-2.7} that 
${^{\pm}}{D}{^{\alpha}}(\mathscr{D}(\Omega)) \subset {^{\pm}}{\mathscr{D}}(\Omega) $,  
Below we show that this inclusion is continuous.

\begin{proposition}\label{continuity}
	 ${^{\pm}}{\mathcal{D}}{^{\alpha}} : \mathscr{D}(\Omega)\to 
	 {^{\pm}}{\mathscr{D}}(\Omega)$ are continuous. 
\end{proposition}

\begin{proof}
We only give a proof for the left derivative ${^{-}}{\mathcal{D}}{^{\alpha}} = {^{-}}{D}{^{\alpha}}$ because the 
other case follows similarly.  

Let $\varphi_k \rightarrow \varphi$ in $\mathscr{D}(\Omega)$,  we want to show that ${^{-}}{{D}}{^{\alpha}} \varphi_k \rightarrow {^{-}}{{D}}{^{\alpha}} \varphi$ in ${^{-}}{\mathscr{D}}(\Omega) $. To the end, let $K \subset\subset  \Omega$ be a compact subset 
so that $\supp(\varphi_k) \subset K$ for every $k \geq 0$ with $\varphi_0\equiv \varphi$, without loss of the generality, assume 
$K=[x_0 , x_1] \subset\subset  \Omega$. Then we have ${^{-}}{D}{^{\alpha}} \varphi_k \equiv 0$ 
for every $x \leq x_0$ and $k \geq 0$, and for any integer $m\geq 0$ and $x > x_0$
        \begin{align*}
            &\bigl|D^m ({^{-}}{D}{^{\alpha}} \varphi_k) (x) - D^m ({^{-}}{D}{^{\alpha}} \varphi)(x)\bigr|\\
            &\quad = \biggl| \dfrac{d^m}{dx^m} \left[ C_\alpha \dfrac{d}{dx} \int_{x_0}^{x} \dfrac{\varphi_k(y)}{(x-y)^{\alpha}} \,dy \right] - \dfrac{d^m}{dx^m} \left[ C_\alpha \dfrac{d}{dx} \int_{x_0}^{x} \dfrac{\varphi(y)}{(x-y)^{\alpha}} \,dy \right] \biggr| \\ 
            & \quad = \left|C_\alpha \dfrac{d^{m+1}}{dx^{m+1}} \int_{x_0}^{x} \dfrac{\varphi_k(y) - \varphi(y)}{(x-y)^{\alpha}}\,dy\right| \\ 
            &\quad = \left|C_\alpha \int_{x_0}^{x} \dfrac{\varphi_k^{(m+1)}(y) - \varphi^{(m+1)}(y)}{(x-y)^{\alpha}}\,dy\right|\\
            &\quad \leq C_\alpha \int_{x_0}^{x_1} 
            \dfrac{ \bigl|\varphi_k^{(m+1)}(y) - \varphi^{(m+1)}(y)\bigr|}{|x-y|^\alpha}\,dy\\
            &\quad \leq \dfrac{|K|^{1-\alpha}C_\alpha}{1-\alpha} \sup_{x\in K}\Bigl|\varphi_k^{(m+1)}(x)- \varphi^{(m+1)}(x)\Bigr|.
        \end{align*}
         
        It follows by the uniform convergence of $\{\varphi_k\}_{k=1}^{\infty}$ that 
        $D^m {^{-}}{{D}}{^{\alpha}}\varphi_k \rightarrow D^m {^{-}}{{D}}{^{\alpha}} \varphi$ uniformly in $\Omega$ for every $m$.  The proof is complete.
    \end{proof}

The above proof also infers that  the  spaces ${^{\pm}}{\mathscr{D}}(\Omega)$ are invariant 
under the mapping ${^{\pm}}{D}{^{\alpha}}$, respectively. 

\begin{proposition}\label{invarance}
	${^{\pm}}{\mathcal{D}}{^{\alpha}} ({^{\pm}}{\mathscr{D}}(\Omega) )\subset 
	{^{\pm}}{\mathscr{D}}(\Omega)$, respectively. Moreover, the inclusion is continuous. 
\end{proposition}

\begin{remark}
	Without any added integrability condition (i.e. decay at $x=\pm \infty$), the inclusions of Proposition \ref{invarance} may not be 
	true when $\Omega=\R$.  The smoothed (in an $\eps$-neighborhood of $x=0$) Heaviside functions $H_\eps(x)$ and $H_\eps (-x)$ are two counterexamples.  In fact, ${^{\pm}}{{D}}{^{\alpha}} \varphi$ may even not exist for some $\varphi\in  
	{^{\pm}}{\mathscr{D}}(\R)$. 
\end{remark}

Let $\mathcal{S}$ denote the space of Schwartz rapidly decaying functions defined in $\R$ (see \cite{Rudin} for the precise definition). Then we have

\begin{lemma}\label{invaranceF}
The space $\mathcal{S}$ is invariant under the Fourier fractional order derivative operator, namely, 	
$ {^{\mathcal{F}}}{D}{^{\alpha}}(\mathcal{S}) \subset \mathcal{S}$.  Moreover, the 
inclusion is continuous. 
\end{lemma}

\begin{proof}
	Let $\varphi\in \mathcal{S}$, it is well known \cite{Adams, Rudin} that $\hat{\varphi}:=\mathcal{F}[\varphi] \in \mathcal{S}$.
	Then $(i\xi)^\alpha \hat{\varphi} \in \mathcal{S}$, so is 
	$ {^{\mathcal{F}}}{D}{^{\alpha}}(\varphi):=\mathcal{F}^{-1} [(i\xi)^\alpha \hat{\varphi}]$. 
	The continuity of the inclusion can be proved in the same way as that in Proposition \ref{continuity}.
\end{proof}

\begin{remark}
	It is easy to check that the Schwartz space $\mathcal{S}$  is not invariant under the mappings
	${^{\pm}}{\mathcal{D}}{^{\alpha}}$, nor is it under ${^{\pm}}{D}{^{\alpha}}$.  Consequently,
	the Fourier fractional derivatives and the Riemann-Liouville fractional derivatives may not coincide 
	for functions in $\mathcal{S}$ in general. 	On the other hand, they do coincide for functions
	in $\mathscr{D}$ (see   (\cite{Feng_Sutton}, Proposition 2.2)). 
	This fact is a main reason for and also validates the choice 	of test functions in the definition of weak fractional derivatives in Section \ref{sec-4.1}. 
\end{remark}

 \subsection{Weak Fractional Derivatives for Compactly Supported Distributions}\label{sec-6.2}
The goal of this subsection is to extend the notion of the   weak fractional derivatives to 
distributions in $\mathscr{D}'$ with compact supports.  First, we recall the definition of
supports for distributions. 

\begin{definition}\label{supportD}
	Let  $u\in \mathscr{D}'(\Omega)$, $u$ is said to vanish on an open subset $O\subset \Omega$
	 if  $u(\varphi)=0$ for all $\varphi\in C^\infty (\Omega)$ with $\supp(\varphi)\subset  O$.  
	 Let $O_{max}$ be a maximal open subset of $\Omega$ on which the distribution $u$ vanishes. 
	 The support of $u$   is defined as the complement of $O_{max}$ in $\Omega$, that is,  
	 $\supp(u):=\Omega\setminus O_{max}$. 
	 Moreover, $u$ is said to be compactly supported if $\supp(u)$  is a compact set. 
\end{definition} 

The best known compactly supported distribution is the Dirac delta function $\delta_0$ which is defined
by $\delta_0(\varphi)=\varphi(0)$ for any $\varphi\in \mathscr{D}(\R)$. $\delta_0$ has the one point
support $\{x=0\}$ and zero order. 

Given a compact subset $K\subset\subset \Omega$, we also define the space 
\[
\mathscr{D}^\prime_K(\Omega):=\bigl\{ u\in \mathscr{D}'(\Omega):\, \supp(u)\subseteq K \bigr\}.
\]

    \begin{lemma}\label{cut_off}
        Let $0 < \alpha < 1$ and $\psi , \varphi \in \mathscr{D}(\Omega) $. Then $\psi {^{\pm}}{D}{^{\alpha}} \varphi \in \mathscr{D}(\Omega)$. Moreover, if $\varphi_k \rightarrow \varphi$ in $\mathscr{D}(\Omega)$, then $\psi {^{\pm}}{D}{^{\alpha}} \varphi_k \rightarrow \psi {^{\pm}}{D}{^{\alpha}} \varphi$ in $\mathscr{D}(\Omega)$.
    \end{lemma}

    \begin{proof}
    	Let  $\psi , \varphi \in \mathscr{D}(\Omega)$.
        Recall that ${^{\pm}}{D}{^{\alpha}} \varphi \in C^{\infty}(\Omega)$. Then, 
         $\psi {^{\pm}}{D}{^{\alpha}} \varphi \in  \mathscr{D}(\Omega)$. 
         It remains to show  the desired convergence result. Again, we only give a proof for 
         the left space because the other case follows similarly. 
    
    Suppose  that $\varphi_k \rightarrow \varphi$ in $\mathscr{D}(\Omega)$, then 
        there exists a compact subset $K \subset \Omega $ so that $\supp( \varphi_k) \subset K$ 
        for all $k\geq 1$.  Without
        loss of the generality, assume $K=[x_0,x_1]$ and $K\cap \supp(\psi)\subset [x_0, x_2]$ for some 
        $x_2> x_0$.  Then, ${^{-}}{D}{^{\alpha}}
        \varphi_k \equiv 0$ for $x\leq x_0$ and all $k\geq 1$ and $\psi\equiv 0$ for all $x>x_2$. 
        Thus,  for any integer $m\geq 1$  and $x_0<x<x_2$
        \begin{align*}
            &\bigl|D^m (\psi {^{-}}{D}{^{\alpha}} \varphi_k (x) )- D^m(\psi {^{-}}{D}{^{\alpha}} \varphi (x) ) \bigr|\\
            &\qquad = \biggl| \sum_{j=0}^{m} \binom{m}{j} \psi^{(m-j)} D^j {^{-}}{D}{^{\alpha}} \varphi_k (x) - \sum_{j=0}^{m} \binom{m}{j} \psi^{(m-j)} D^j {^{-}}{D}{^{\alpha}} \varphi (x) \biggr| \\ 
            &\qquad \leq \sum_{j=0}^{m} \binom{m}{j} \left|\psi^{(m-j)}(x) D^j {^{-}}{D}{^{\alpha}} (\varphi_k - \varphi)(x)\right|\\
            &\qquad = \sum_{j=0}^{m} \binom{m}{j} \left|\psi^{(m-j)}(x) {^{-}}{I}{^{1-\alpha}}(\varphi_{k}^{(j+1)} - \varphi^{(j+1)})(x)\right|\\
            &\qquad \leq \dfrac{C_m C_\alpha}{1-\alpha} \sup_{x_0\leq x \leq x_2 \atop 1\leq j\leq m}  \Bigl( \bigl|\psi^{(m-j)}(x)\bigr|\cdot \bigl|\varphi_k^{(j+1)}(x) - \varphi^{(j+1)}(x) \bigr|  \Bigr).
        \end{align*}
        By the uniform convergence of $\varphi_k$ to $\varphi$ in $\mathscr{D}(\Omega)$, we obtain the desired result. 
    \end{proof}

 The above lemma guarantees that  $\psi {^{\pm}}{D}{^{\alpha}} \varphi$ belongs to the 
 standard test space $ \mathscr{D}(\Omega)$  which removes most pollution contribution in  
 ${^{\pm}}{D}{^{\alpha}} \varphi$ by using a compactly supported smooth (cutoff) function $\psi$. 

For compactly supported distributions, there holds the following result, its proof can be 
found in \cite[Theorem 6.24]{Rudin}.

\begin{theorem}\label{extension_rudin}
	Let $u\in \mathscr{D}^\prime_K(\Omega)$. 
	Then $u$ has a finite (integer) order $N (\geq 0)$ and can be uniquely extended to a 
	continuous  linear functional on $C^\infty(\Omega)$ which is given by 
	\[
	\tilde{u}(\varphi) := u(\psi \varphi) \qquad \forall \varphi\in C^\infty(\Omega),
	\] 
	where $\psi\in C^\infty_0(\Omega)$ satisfying $\psi\equiv 1$ in $K$ is a partition of unity.
\end{theorem}

We note that the extension $\tilde{u}$ as a functional does not depend on the choice of 
the cut-off function $\psi$ (see \cite{Rudin} for a proof). 

We now are ready to define weak fractional derivatives for compactly supported distributions in $\mathscr{D}'(\Omega)$.

    \begin{definition}
        Let $ \alpha >0$ and  $u \in \mathscr{D}^\prime_K(\Omega)$. 
        Define ${^{\pm}}{\mathcal{D}}{^{\alpha}} u:\mathscr{D}(\Omega)\to \R$  respectively by  
        \begin{align*}
           {^{\pm}}{\mathcal{D}}{^{\alpha}} u  (\varphi):= (-1)^{[\alpha]} \tilde{u} ({^{\mp}}{D}{^{\alpha}} \varphi)   
                    = (-1)^{[\alpha]} u(\psi {^{\mp}}{D}{^{\alpha}} \varphi)   
           \qquad \forall \varphi \in \mathscr{D}(\Omega),
        \end{align*}
        where $\tilde{u}$ and $\psi$ are the same as in Theorem \ref{extension_rudin}. 
        
    \end{definition}
 
 The next theorem shows that a compactly supported distribution $u\in \mathscr{D}^\prime_K(\Omega)$ has 
 any order  weak fractional derivative ${^{\pm}}{\mathcal{D}}{^{\alpha}} u$ which belongs to $\mathscr{D}'(\Omega)$.
 
    \begin{theorem}\label{existence}
        Let $\alpha>0$ and suppose $u \in \mathscr{D}^\prime_K(\Omega)$. Then 
        \begin{itemize}
       \item [{\rm (i)}] ${^{\pm}}{\mathcal{D}}{^{\alpha}} u \in \mathscr{D}'(\Omega)$. Moreover, if $K \subseteq [c,d]\subset\subset\Omega$, then $\supp({^{-}}{\mathcal{D}}{^{\alpha}} u)\subseteq(-\infty, d]\cap \Omega$ and $\supp({^{+}}{\mathcal{D}}{^{\alpha}} u)\subseteq[c, \infty)\cap \Omega$.
       \item[{\rm (ii)}] Suppose that $\{ u_j\}_{j=1}^{\infty} \subset \mathscr{D}^\prime_K(\Omega)$
       such that 
        $u_j \rightarrow u$ in $\mathscr{D}^\prime_K(\Omega)$, then ${^{\pm}}{\mathcal{D}}{^{\alpha}} u_j \rightarrow {^{\pm}}{\mathcal{D}}{^{\alpha}} u$ in $\mathscr{D}'(\Omega)$.
        \end{itemize}
    \end{theorem}

    \begin{proof}
       (i) The linearity of ${^{\pm}}{\mathcal{D}}{^{\alpha}} u$ is trivial. 
       To show the continuity, it suffices to show that ${^{\pm}}{\mathcal{D}}{^{\alpha}} u$ 
       is sequentially continuous at zero. To the end, let $\{\varphi_k\}_
       {k=1}^{\infty} \subset \mathscr{D}(\Omega)$ so that 
       $\varphi_k \rightarrow 0$ in $\mathscr{D}(\Omega)$. It follows by Lemma \ref{cut_off} that 
        \begin{align*}
           {^{\pm}}{\mathcal{D}}{^{\alpha}}u(\varphi_k) = (-1)^{[\alpha]} u(\psi {^{\mp}}{D}{^{\alpha}} \varphi_k) \to   u(\psi {^{\mp}}{D}{^{\alpha}} (0)) = 0
           \qquad\mbox{as }k\to \infty.
        \end{align*}
        
        Since for any $\varphi \in  \mathscr{D}(\Omega)$, ${^{\pm}}{\mathcal{D}}{^{\alpha}}\varphi
        \in {^{\pm}}{C}{^{\infty}_{0}}(\Omega)$, then the supports of ${^{\pm}}{\mathcal{D}}{^{\alpha}}u$
        pollute that of $u$ to the right/left accordingly.  
        
        (ii) Suppose that $u_j \to u$ in $\mathscr{D}'(\Omega)$, we have for any $\varphi\in \mathscr{D}(\Omega)$
        \begin{align*}
        {^{\pm}}{\mathcal{D}}{^{\alpha}} u_j (\varphi) : = (-1)^{[\alpha]} u_j (\psi {^{\mp}}{D}{^{\alpha}} \varphi) 
        \underset{j\to \infty}{\longrightarrow}  (-1)^{[\alpha]} u( \psi {^{\mp}}{D}{^{\alpha}} \varphi) 
        ={^{\pm}}{\mathcal{D}}{^{\alpha}} u(\varphi).
        \end{align*}
        Thus, ${^{\pm}}{\mathcal{D}}{^{\alpha}} u_j \rightarrow {^{\pm}}{\mathcal{D}}{^{\alpha}} u$ in $\mathscr{D}'(\Omega)$ as $j\to\infty$.
        The proof is complete.
    \end{proof}

    \begin{proposition}
        Let $\alpha>0$ and suppose $u \in \mathscr{D}^\prime_K(\Omega)$. Then ${^{\pm}}{\mathcal{D}}{^{\alpha}} u  \to \mathcal{D}u $ in $\mathscr{D}'(\Omega)$ as $\alpha \rightarrow 1^-$ and ${^{\pm}}{\mathcal{D}}{^{\alpha}} u  \to \mathcal{D}u $ in $\mathscr{D}'(\Omega)$ as $\alpha \rightarrow 1^+$.
    \end{proposition}

    \begin{proof}
        For any $\varphi \in C^{\infty}_{0}(\Omega)$, we have 
        \begin{align*}
             {^{\pm}}{\mathcal{D}}{^{\alpha}} u (\varphi) = (-1)^{[\alpha]}   u(\psi {^{\mp}}{D}{^{\alpha}} \varphi) \underset{\alpha\to 1^-}{\longrightarrow}  -u(\psi D \varphi)= -u(D\varphi) = \mathcal{D}u(\varphi),\\
              {^{\pm}}{\mathcal{D}}{^{\alpha}} u (\varphi) = (-1)^{[\alpha]}   u(\psi {^{\mp}}{D}{^{\alpha}} \varphi) \underset{\alpha\to 1^+}{\longrightarrow}  - u(\psi D \varphi)=-u(D\varphi) =\mathcal{D}u(\varphi). 
        \end{align*}
        Hence, the assertions hold.
    \end{proof}

    \begin{proposition}
        Let $\Omega = (a,b)$ and $0<\alpha<1$. Suppose that $ u \in \mathscr{D}^\prime_K(\Omega)$ and $\eta \in C^{\infty}(\Omega)$, 
        then there holds the following product rule: 
        \begin{align}\label{product_rule_dist}
         {^{\pm}}{\mathcal{D}}{^{\alpha}} (\eta u) 
         = \eta {^{\pm}}{\mathcal{D}}{^{\alpha}} u -  \sum_{k=1}^{m} D^k \eta \, {^{\pm}}{I}{^{k-\alpha}} u
         - C_{m,\alpha}  (\eta^{(m+1)} * \mu_{\pm}) \,  (\kappa_{\pm}^{-\alpha} *\psi u),
        \end{align}
        where
        \begin{align} \label{C_kalpha}
        C_{k,\alpha} &:= \dfrac{\Gamma(1+\alpha)}{\Gamma(k+1) \Gamma(1-k + \alpha)},\\
        {^{\pm}}{I}{^{k-\alpha}} u (\varphi) &:= C_{k,\alpha} u\bigl(\psi {^{\mp}}{I}{^{k-\alpha}} \varphi \bigr)
        \qquad \forall \varphi\in \mathscr{D}(\Omega). \label{fractional_int_dist}
        \end{align}
       
    \end{proposition}

    \begin{proof} By the fractional order product rule, we have 
        \begin{align*}
            {^{\pm}}{\mathcal{D}}{^{\alpha}} (\eta u)(\varphi) 
            :& = \eta u ( \psi {^{\mp}}{D}{^{\alpha}} \varphi ) 
            =   u( \eta \psi {^{\mp}}{D}{^{\alpha}} \varphi ) 
            = (u , \psi \eta {^{\mp}}{D}{^{\alpha}} \varphi ) \\
            &= u\bigl(\psi {^{\mp}}{D}{^{\alpha}} (\varphi \eta) \bigr)  - u \Bigl( \psi \sum_{k=1}^{m} C_{k,\alpha} {^{\mp}}{I}{^{k-\alpha}} \varphi D^{k} \eta \Bigr) 
            - u \bigl(\psi {^{\mp}}{R}{^{\alpha}}(\varphi, \eta) \bigr) \\
            &=: I - II - III
        \end{align*}
        where 
        \begin{align*}
        {^{+}}{R}{^{\alpha}_{m}}(\varphi, \eta) &: = \dfrac{(-1)^{m+1}}{m! \Gamma(-\alpha)} \int_{x}^{b} \dfrac{\varphi(y)}{(y-x)^{1+\alpha}}\,dy \int_{x}^{y} \eta^{(m+1)}(z) (z-x)^{m} \,dz 
        \end{align*}
        with a similar formula for ${^{-}}{R}{^{\alpha}_{m}} (\varphi, \eta)$.  
        
        For terms $I$ and $II$  we have 
        \begin{align*}
            I:& =  u\Bigl(\psi {^{\mp}}{D}{^{\alpha}} (\varphi \eta ) \Bigr)  
            ={^{\pm}}{\mathcal{D}}{^{\alpha}}u(\eta \varphi) 
            = \eta {^{\pm}}{\mathcal{D}}{^{\alpha}} u(\varphi),\\
           II:&= u\Bigl( \psi \sum_{k=1}^{m} C_{k,\alpha} {^{\mp}}{I}{^{k-\alpha}} \varphi D^{k} \eta \Bigr)
            = \sum_{k=1}^{m} C_{k,\alpha} u\bigl(\psi {^{\mp}}{I}{^{k-\alpha}} \varphi D^{k} \eta \bigr) \\ 
           &= \sum_{k=1}^{m} C_{k,\alpha} D^{k} \eta \, u\bigl(\psi {^{\mp}}{I}{^{k-\alpha}} \varphi \bigr) 
          = \sum_{k=1}^{m} D^k \eta \, {^{\pm}}{I}{^{k-\alpha}} u (\varphi).
       \end{align*}

        Finally, to simplify term $III$, we rewrite the remainder formula as follows: 
        \begin{align*}
            {^{+}}{R}{^{\alpha}_{m}}(\varphi, \eta) 
            &: = \dfrac{(-1)^{m+1}}{m! \Gamma(-\alpha)} \int_{x}^{b} \int_{x}^{y} \dfrac{\varphi(y)}{(y-x)^{1+\alpha}} \eta^{(m+1)}(z)(z-x)^{m} \,dzdy \\ 
            &= \dfrac{(-1)^{m+1}}{m!\Gamma(-\alpha)} \int_{x}^{b} \dfrac{\varphi(y)}{(y-x)^{1+\alpha}} (\eta^{(m+1)} * \mu_{+})(y)\,dy \\ 
            &= C_{m , \alpha} \bigl( \varphi (\eta^{(m+1)} * \mu_{+}) * \kappa_+^{-\alpha} \bigr)(x).
        \end{align*}
        Then we have 
        \begin{align*}
            III :&= u\bigl( \psi {^{+}}{R}{^{\alpha}}(\varphi, \eta) \bigr) \\ 
            &=  u \bigl( C_{m,\alpha} \psi (\varphi (\eta^{(m+1)} * \mu_{+} ) *\kappa^{-\alpha}_{+}) \bigr) \\
            &= C_{m,\alpha} u \bigl(\psi (\varphi (\eta^{(m+1)} * \mu_{+} ) * \kappa^{-\alpha}_{+}) \bigr) \\ 
            &=  C_{m,\alpha}  (\kappa_{+}^{-\alpha} * \psi u) \bigl( \varphi (\eta^{(m+1)} *\mu_+) \bigr)  \\
            &=C_{m,\alpha}  (\eta^{(m+1)} * \mu_{+}) \,  (\kappa_{+}^{-\alpha} * \psi u) ( \varphi).  
            &= C_{m,\alpha} (\eta^{(m+1)} * \mu_{+}) \cdot (\kappa_{+}^{-\alpha} *(\psi u)) (\varphi). 
        \end{align*}
   The desired formula \eqref{product_rule_dist} follows from combining the above three identities. The proof is complete. 
    \end{proof}

  \subsection{Weak Fractional Derivatives for Distributions on Finite Intervals}\label{sec-6.3}
  In the previous subsection we introduce a fractional order derivative notion for compactly supported 
  distributions in $\mathscr{D}_K^\prime(\Omega)$. The aim of this subsection is to introduce a fractional 
  derivative notion for general distributions in $\mathscr{D}^\prime(\Omega)$ when $\Omega=(a,b)$ is finite. 
  We shall address the case $\Omega=\R$ in the next subsection. 
  
First, we consider the class of one sided generalized functions in 
${^{\pm}}{\mathscr{D}}^\prime (\Omega):= ({^{\pm}}{\mathscr{D}} (\Omega ))^\prime$,
which are proper subspaces of ${\mathscr{D}}^\prime (\Omega)$. By Proposition \ref{invarance} we know that
${^{\pm}}{\mathscr{D}}(\Omega)$ are respectively invariant under the mappings ${^{\pm}}{D}{^{\alpha}}$. 
This fact then makes defining ${^{\pm}}{\mathcal{D}}{^{\alpha}}u$
for $u\in {^{\pm}}{\mathscr{D}}^\prime (\Omega)$ a trivial task.

 \begin{definition}
	Let $ \alpha >0$ and  $u \in {^{\pm}}{\mathscr{D}}^\prime (\Omega) $. 
	Define ${^{\pm}}{\mathcal{D}}{^{\alpha}} u: {^{\pm}}{\mathscr{D}}(\Omega) \to \R$ respectively by  
	\begin{align}\label{eq6.3}
	{^{\pm}}{\mathcal{D}}{^{\alpha}} u  (\varphi):=  
	 (-1)^{[\alpha]} u({^{\mp}}{D}{^{\alpha}} \varphi)   
	\qquad \forall \varphi \in {^{\pm}}{\mathscr{D}}(\Omega) .
	\end{align}
  
\end{definition}

Clearly, ${^{\pm}}{\mathcal{D}}{^{\alpha}} u$ is well defined and ${^{\pm}}{\mathcal{D}}{^{\alpha}} u
\in {^{\pm}}{\mathscr{D}}^\prime (\Omega) $, respectively. It also can be shown that many other properties 
hold for the fractional order derivative operators ${^{\pm}}{\mathcal{D}}{^{\alpha}}$ on the one sided 
generalized function spaces ${^{\pm}}{\mathscr{D}}^\prime (\Omega)$. We leave the verification to 
the interested reader.   

To define fractional order derivatives for distributions in ${\mathscr{D}}^\prime (\Omega)\setminus {^{-}}{\mathscr{D}}^\prime (\Omega)\cup {^{+}}{\mathscr{D}}^\prime (\Omega)$, we need to 
construct``good" extensions for any distribution $u\in {\mathscr{D}}^\prime (\Omega)$ to 
${^{-}}{\mathscr{D}}^\prime (\Omega)$ and  ${^{+}}{\mathscr{D}}^\prime (\Omega)$. 
This will be done below by using the partition of unity theorem to define 
$u (\varphi):= \sum_{j = 1}^{\infty} u( \psi_j  \varphi)$ for any $\varphi\in {^{\pm}}{\mathscr{D}}(\Omega)$.

Let $\{I_\beta\}$ be a family of open subintervals of $(a,b)$ which forms a covering of $\Omega$. 
By the  partition of unity theorem (cf. \cite{Rudin}),  there exists a subsequence $\{I_j\}_{j =1}^{\infty} \subset \{I_\beta\}$ and a partition  of the unity $\{\psi_j\}_{j = 1}^{\infty}$
subordinated to $\{I_j\}_{j= 1}^{\infty}$, namely, $\psi_j\in C^\infty_0(I_j)$ for $j\geq 1$ and $\sum \psi_j(x)\equiv 1$ 
on every compact subset $K$ of $\Omega$ and the sum is a finite sum for every $x\in K$.

\begin{definition}\label{def-partition}
	Let $ \alpha >0$ and  $u \in  {\mathscr{D}}^\prime (\Omega) $. 
	Define ${^{\pm}}{\mathcal{D}}{^{\alpha}} u:  {\mathscr{D}}(\Omega) \to \R$ respectively by  
	\begin{align}\label{eq6.4}
	{^{\pm}}{\mathcal{D}}{^{\alpha}} u  (\varphi):=  
	(-1)^{[\alpha]} \sum_{j = 1}^{\infty} u( \psi_j {^{\mp}}{D}{^{\alpha}} \varphi)   
	\qquad \forall \varphi \in  {\mathscr{D}}(\Omega) .
	\end{align}

\end{definition}

We claim that ${^{\pm}}{\mathcal{D}}{^{\alpha}} u$ is well defined and ${^{\pm}}{\mathcal{D}}{^{\alpha}} u \in  {\mathscr{D}}^\prime (\Omega)$, respectively.
We leave the verification to the interested reader.

  \subsection{Weak Fractional Derivatives for Distributions on $\R$}\label{sec-6.4}
  To define Riemann-Liouville fractional order derivatives for distributions in ${\mathscr{D}}^\prime (\R)$ is more 
  complicated than in ${^{\pm}}{\mathscr{D}}^\prime (\Omega)$; the complication is due
  to the fact that the kernel functions $\kappa_{\pm}^{\alpha} \not\in L^1(\R)$ and the pollutions of ${^{\pm}}{\mathcal{D}}{^{\alpha}}\varphi(x)$ for $\varphi\in \mathscr{D}(\R)$ do not decay fast enough 
  when $x\to \pm \infty$. 
  
  We first consider the simpler case of Fourier fractional order derivatives for tempered distributions 
  in $\mathcal{S}^\prime(\R)$. By Proposition \ref{invaranceF} we know that the Schwartz space $\mathcal{S}(\R)$ 
  is invariant under the Fourier derivative operator ${^{\mathcal{F}}}{D}{^{\alpha}}$. This allows us easily 
  to define  Fourier fractional derivatives for tempered distributions as follows.
  
  \begin{definition}
 	Let $ \alpha >0$ and  $u \in  {\mathcal{S}}^\prime (\R) $. 
 	Define ${^{\mathcal{F}}}{\mathcal{D}}{^{\alpha}} u:  \mathcal{S}(\R) \to \R$ by  
 	\begin{align}\label{eq6.5}
 	 {^{\mathcal{F}}}{\mathcal{D}}{^{\alpha}} u  (\varphi):=  
 	(-1)^{[\alpha]} u \bigl( {^{\mathcal{F}}}{D}{^{\alpha}} \varphi \bigr)  
 	\qquad \forall \varphi \in  \mathcal{S}(\R).
 	\end{align} 
\end{definition}
  
It is easy to verify that ${^{\mathcal{F}}}{\mathcal{D}}{^{\alpha}} u$ is well defined and 
${^{\mathcal{F}}}{\mathcal{D}}{^{\alpha}} u \in \mathcal{S}^\prime(\R)$. It also can be shown that 
many other properties hold for the fractional order derivative operator ${^{\mathcal{F}}}{\mathcal{D}}{^{\alpha}}$ 
on the space of tempered distributions ${\mathcal{S}}^\prime (\R)$. We leave the verification to 
the interested reader.

To define fractional order derivatives for distributions in
${\mathscr{D}}^\prime (\R)\setminus {\mathcal{S}}^\prime (\R)$, we need to extend the domain 
of $u\in {\mathscr{D}}^\prime (\R)$ from ${\mathscr{D}}(\R)$ to ${\mathcal{S}}(\R)$ (or  ${^{\pm}}{\mathscr{D}}(\R)$). 
Again, this will be done by using the partition of the unity theorem as seen above 
to define 
$u (\varphi):= \sum_{j = 1}^{\infty} u( \psi_j  \varphi)$ for any $\varphi \in  {^{\pm}}{\mathscr{D}}(\R)$.

Let $\{I_\beta\}$ be a family of open finite subintervals of $\R$ which forms a covering of $\R$. 
By the  partition of unity theorem (cf. \cite{Rudin}),  there exists a subsequence $\{I_j\}_{j =1}^{\infty} \subset \{I_\beta\}$ and a partition  of the unity $\{\psi_j\}_{j = 1}^{\infty}$
subordinated to $\{I_j\}_{j= 1}^{\infty}$, namely, $\psi_j\in C^\infty_0(I_j)$ 
for $j\geq 1$ and $\sum \psi_j(x)\equiv 1$ on every compact subset $K$ of $\R$ 
and the sum is a finite sum for every $x\in K$.

	\begin{definition}\label{def-partition_R}
		Let $ \alpha >0$ and  $u \in  {\mathscr{D}}^\prime (\R) $. 
		Define ${^{\pm}}{\mathcal{D}}{^{\alpha}} u:  {\mathscr{D}}(\R) \to \R$ respectively by  
		\begin{align}\label{eq6.6}
		{^{\pm}}{\mathcal{D}}{^{\alpha}} u  (\varphi):=  
		(-1)^{[\alpha]} \sum_{j = 1}^{\infty} u( \psi_j {^{\mp}}{D}{^{\alpha}} \varphi)   
		\qquad \forall \varphi \in  {\mathscr{D}}(\R) .
		\end{align}

	\end{definition} 

 We claim that ${^{\pm}}{\mathcal{D}}{^{\alpha}} u$ is well defined and ${^{\pm}}{\mathcal{D}}{^{\alpha}} u \in  {\mathscr{D}}^\prime (\R) $, respectively.
Again, we leave the verification to the interested reader.

\section{Conclusion}\label{sec-7}
 In this paper we first recalled various definitions of classical fractional 
  	derivatives and gave a new interpretation of the classical theory from a different 
   perspective, and especially emphasized the importance of the Fundamental 
   Theorem of Classical  Fractional Calculus (FTcFC) and its ramifications in the classical theory. 
   We then presented a self-contained new theory of weak fractional differential 
   calculus. The crux of this new theory 
   is the introduction of a weak fractional derivative notion which is a natural 
   generalization of integer order weak derivatives; it also helps to unify multiple existing 
   fractional derivative definitions and has the potential to be easily extended to higher dimensions.  Various calculus rules including a Fundamental Theorem of Weak 
   Fractional Calculus (FTwFC), product and chain rules,
   and integration by parts formulas were established for weak fractional derivatives 
   and relationships with existing classical 
   fractional derivatives were also obtained.   This weak fractional differential calculus 
   theory lays down the ground work for developing a new fractional order Sobolev space theory   in a companion paper \cite{Feng_Sutton1a}. 
   Furthermore, the notion of weak fractional derivatives was systematically extended 
   to general distributions instead of only to some special distributions as done in the literature. 
   It is expected (and our hope, too) that these newly developed theories of weak fractional 
   differential calculus and fractional order Sobolev spaces will lay down a solid theoretical  
   foundation for systematically and rigorously developing a fractional calculus of variations 
   theory and a fractional PDE theory as well as their numerical solutions. Moreover, we hope this work  will stimulate more research on and applications of fractional calculus and 
   fractional differential equations in the near future.



\begin{thebibliography}{99}
\bibliographystyle{abbrv}

\bibitem{Adams}
{\sc R.A.~Adams}, 
{\em Sobolev Spaces}, 
Pure and  Applied Mathematics, Vol. 65. Academic Press, New York, 1975.

\bibitem{Bassam}
{\sc M.A. Bassam},
{\em Some properties of Holmgren-Riesz transform},
Ann. Scuola Norm. Super., Pisa,  1961.

\bibitem{Brezis}
{\sc H. Brezis},
{\em Functional Analysis, Sobolev Spaces and Partial Differential Equations},
Springer, New York, 2011.

\bibitem{Das}
{\sc S. Das},
{\em Functional Fractional Calculus},
Springer, Berlin, 2011.

\bibitem{Deng}
{\sc C. Li and  W. Deng},
{\em Remarks on fractional derivatives},
Appl. Math. Comput., 187(2), 777-784, 2007.


\bibitem{Du2019}
{\sc Q. Du}, 
{\em Nonlocal Modeling, Analysis, and Computation}, 
SIAM, Philadelphia, 2019.

\bibitem{Ervin}
{\sc V. Ervin and J. P. Roop},
{\em Variational formulation for the stationary fractional advection dispersion equation},
Numer.  Methods for PDEs,  22(2), 558-576, 2006.

\bibitem{Evans} 
{\sc L. C. Evans}, 
{\em  Partial Differential Equations},  AMS, Providence, RI,  2010.

\bibitem{Feng_Sutton}
{\sc X. Feng and M. Sutton},
{\em A new theory of fractional differential calculus and fractional Sobolev spaces: One-dimensional case},
arXiv:2004.10833.

\bibitem{Feng_Sutton1a}
{\sc X. Feng and M. Sutton},
{\em New fractional Sobolev spaces in one dimension},
in preparation.

\bibitem{Feng_Sutton2}
{\sc X. Feng and M. Sutton},
{\em  A new theory of fractional calculus of variations and fractional differential equations},
in preparation.

\bibitem{Feng_Sutton3}
{\sc X. Feng and M. Sutton},
{\em  Finite element methods for approximating weak fractional  derivatives and fractional differential equations},
in preparation.

\bibitem{Trudinger}
{\sc D. Gilbarg and N. Trudinger},
{\em Elliptic Partial Differential Equations of Second Order},
Springer, New York, 1998.

\bibitem{Guo}
{\sc B. Guo, X. Pu, and  F. Huang},
{\em Fractional Partial Differential Equations and Their Numerical Solutions}, 
World Scientific Publishing Co.,  London,  2015.

\bibitem{Herzallah}
{\sc M. Herzallah},
{\em Notes on some fractional calculus operators and their properties},
J. Fract. Calc.  Appli, 5(3s), 1-10, 2014.

\bibitem{Hilfer}
{\sc R. Hilfer},
{\em  Applications of Fractional Calculus in Physics},
World Scientific Press, 2000.


\bibitem{Khalil2014}
{\sc R. Khalil, M.Al Horani, A. Youse,and M. Sababheh},
{\em A new definition of fractional derivative},
J. Computat. Applied Math,  264, 65--70, 2014.


\bibitem{Kilbas}
{\sc A. A. Kilbas, H. M. Srivastava, and J. J. Trujillo},
{\em Theory and Applications of Fractional Differential Equations}, 
Elsevier, 2006. 	

\bibitem{Klimek}
{\sc M. Klimek},
{\em On Solutions of Linear Fractional Differential Equations of Variational Type},
The Publishing Office of Czestochowa University of Technology, 2009.

\bibitem{Li}
    {\sc C. Li, D. Qian, and  Y. Chen},
    {\em On Riemann-Liouville and Caputo derivatives},
   Discet. Dyn. in Nature and Society, 1-15,  2011.

\bibitem{Liouville}
    {\sc J. Liouville},
    {\em Memoire sur le calcul des differentielles a indices quelconques},
    Journal de l'Ecole Royale Polytechnique, Extraits du Tome,  13. Sect. 21, 71 - 162, 1832.


\bibitem{Malinowska}
{\sc A. B. Malinowska, T. Odzijewicz, and D. F. M. Torres},
  {\em Advanced Methods in the Fractional Calculus of Variations},
   Springer, Berlin,  2015.


\bibitem{Meerschaert}
{\sc M. Meerschaert and A. Sikorskii},
  {\em Stochastic Models for Fractional Calculus},   de Gruyter,  2012.
   
\bibitem{Munkhammar}
{\sc J. Munkhammar},
{\em Riemann-Liouville Fractional derivatives and the Taylor-Riemann series},
 2004
 
\bibitem{Meyers}
 {\sc N. G. Meyers and J. Serrin},
 {\em H = W}, 
Proceed. Nati. Acad. of Sci., 51(6), 1055-1056, 1964.

\bibitem{Nezza}
    {\sc E. Di Nezza, G. Palatucci, and  E. Valdinoci},
    {\em Hitchhiker's guide to the fractional Sobolev spaces},
    Bulletin des sciences mathematique, 136(5), 521-573, 2012. 
    
\bibitem{Osler}
{\sc T. Osler},
{\em Leibniz rule for fractional derivatives generalized and an application to infinite series},
SIAM J. Appl. Math., 18(3), 658-674, 1970.
    
   
\bibitem{Podlubny}
{\sc I. Podlubny},
{\em Fractional Differential Equations}, 
Mathematics in science and engineering, Vol. 198, Academic Press,  New York, 1999.
    
    
\bibitem{Rudin}
{\sc W. Rudin},
{\em Fractional Analysis},  McGraw-Hill, New York, 1991.   
    
    
\bibitem{Samko}
{\sc S. G. Samko, A. A. Kilbas, and O. I. Marichev},
{\em Fractional Integrals and Derivatives: Theory and Applications},
GRC Press,  1993.

\bibitem{Stinga}
{\sc P. R. Stinga and M. Vaughan}
{\em One-sided fractional derivatives, fractional Laplacians, and weighted Sobolev spaces},
Nonl. Anal., 193, https://doi.org/10.1016/j.na.2019.04.004, 2020.
 
    
\bibitem{Tarasov}
{\sc V. Tarasov},
{\em No violation of the Leibniz rule. No fractional derivative},
Comm. Nonl. Sci. Numer. Simul., 18, 2945-2948, 2013.

\bibitem{Tarasov2}
{\sc V. Tarasov},
{\em On chain rule for fractional derivatives.},
Comm. Nonl. Sci. Numer. Simul., 30, 1-3, 2016.

\end{thebibliography}
\end{document}